\documentclass[oneside,american,british,english]{amsart}
\usepackage[T1]{fontenc}
\usepackage[latin9]{inputenc}
\usepackage{float}
\usepackage{mathtools}
\usepackage{amstext}
\usepackage{amsthm}
\usepackage{amssymb}

\makeatletter

\providecommand{\tabularnewline}{\\}

\numberwithin{equation}{section}
\numberwithin{figure}{section}
\theoremstyle{plain}
\newtheorem{thm}{\protect\theoremname}[section]
\theoremstyle{plain}
\newtheorem{lem}[thm]{\protect\lemmaname}
\theoremstyle{definition}
\newtheorem*{defn*}{\protect\definitionname}
\theoremstyle{remark}
\newtheorem*{rem*}{\protect\remarkname}
\theoremstyle{definition}
\newtheorem*{example*}{\protect\examplename}
\theoremstyle{plain}
\newtheorem*{lem*}{\protect\lemmaname}
\theoremstyle{plain}
\newtheorem{prop}[thm]{\protect\propositionname}
\theoremstyle{plain}
\newtheorem{cor}[thm]{\protect\corollaryname}

\usepackage{colortbl}
\usepackage{graphicx}
\usepackage{pgf,tikz,tkz-graph,subcaption}

\usetikzlibrary{arrows,shapes}
\usetikzlibrary{decorations.pathreplacing}

\makeatother

\usepackage{babel}
\addto\captionsamerican{\renewcommand{\corollaryname}{Corollary}}
\addto\captionsamerican{\renewcommand{\definitionname}{Definition}}
\addto\captionsamerican{\renewcommand{\examplename}{Example}}
\addto\captionsamerican{\renewcommand{\lemmaname}{Lemma}}
\addto\captionsamerican{\renewcommand{\propositionname}{Proposition}}
\addto\captionsamerican{\renewcommand{\remarkname}{Remark}}
\addto\captionsamerican{\renewcommand{\theoremname}{Theorem}}
\addto\captionsbritish{\renewcommand{\corollaryname}{Corollary}}
\addto\captionsbritish{\renewcommand{\definitionname}{Definition}}
\addto\captionsbritish{\renewcommand{\examplename}{Example}}
\addto\captionsbritish{\renewcommand{\lemmaname}{Lemma}}
\addto\captionsbritish{\renewcommand{\propositionname}{Proposition}}
\addto\captionsbritish{\renewcommand{\remarkname}{Remark}}
\addto\captionsbritish{\renewcommand{\theoremname}{Theorem}}
\addto\captionsenglish{\renewcommand{\corollaryname}{Corollary}}
\addto\captionsenglish{\renewcommand{\definitionname}{Definition}}
\addto\captionsenglish{\renewcommand{\examplename}{Example}}
\addto\captionsenglish{\renewcommand{\lemmaname}{Lemma}}
\addto\captionsenglish{\renewcommand{\propositionname}{Proposition}}
\addto\captionsenglish{\renewcommand{\remarkname}{Remark}}
\addto\captionsenglish{\renewcommand{\theoremname}{Theorem}}
\providecommand{\corollaryname}{Corollary}
\providecommand{\definitionname}{Definition}
\providecommand{\examplename}{Example}
\providecommand{\lemmaname}{Lemma}
\providecommand{\propositionname}{Proposition}
\providecommand{\remarkname}{Remark}
\providecommand{\theoremname}{Theorem}

\begin{document}
\title{Extrema of Local Mean and Local Density in A Tree}
\author{Ruoyu Wang}
\email{ruoyu.wang@math.uu.se}
\address{Department of Mathematics, Uppsala University}
\begin{abstract}
Given a tree $T,$ one can define the local mean at some subtree $S\subseteq T$
to be the average order of subtrees containing $S.$ It is natural
to ask which subtree of order $k$ achieves the maximal/minimal local
mean among all the subtrees of the same order and what properties
it has. We call such subtrees $k$-maximal subtrees. Wagner and Wang
showed in 2016 that a 1-maximal subtree is a vertex of degree 1 or
2. This paper shows that for any integer $k=1,\ldots,\left|T\right|,$
a $k$-maximal subtree has at most one leaf whose degree is greater
than 2 and at least one leaf whose degree is at most 2. Furthermore,
we show that a $k$-maximal subtree has a leaf of degree greater than
2 only when all its other leaves are leaves in $T$ as well.

In the second part, this paper introduces the local density as a normalization
of local means, for the sake of comparing subtrees of different orders,
and shows that the local density at subtree $S$ is lower-bounded
by $1/2$ with equality if and only if $S$ contains the core of $T.$
On the other hand, local density can be arbitrarily close to 1.
\end{abstract}

\maketitle

\section{Introduction}

In the early 1980s, Jamison \cite{MR735190,MR762896} initiated the
study of mean order of subtrees and density of a tree $T.$ There
are two types of means studied therein. First is the global mean,
which is the average order of $\mathit{all}$ subtrees of $T.$ Then
density is defined as global mean divided by $\left|T\right|$ (known
as the order of $T,$ equal to the cardinality of the underlying vertex
set). The other type is the local mean at a subtree $S\subseteq T,$
where only subtrees containing $S$ are taken into consideration,
i.e., the average order of all subtrees $\mathit{containing}$ subtree
$S.$ Local means can be compared among all the subtrees of the same
order. Throughout this paper, a $k$-maximal subtree is a subtree
of order $k$ that attains the maximal local mean among all the subtrees
of order $k.$

In \cite{MR735190,MR762896}, among other results, Jamison laid the
groundwork for subsequent research in this area:
\begin{itemize}
\item global mean is no greater than local mean at a vertex,
\item local mean manifests monotonicity with respect to set inclusion (i.e.,
for subtrees $S\subsetneq S',$ local mean at $S$ is greater than
local mean at $S'$),
\item the local mean at $S$ is lower bounded by $(|S|+|T|)/2,$ etc. 
\end{itemize}
Several conjectures and open questions were raised in these two papers,
most of which have been solved successively during the 2010s. In 2010,
Vince and Wang proved in \cite{MR2595700} that for a tree with no
vertex of degree 2, the local mean is at least half of $\left|T\right|$
and less than $\frac{3}{4}$ of $\left|T\right|,$ which proves the
conjecture \cite[(7.2)]{MR735190}. Later on, in 2014, \foreignlanguage{british}{Haslegrave}
provided characterizations for both bounds in \cite{MR3213626}. Wagner
and Wang \cite{MR3433637} studied the ratio of local mean (at a vertex)
to global mean and showed that it is bounded above by 2, proving the
conjecture \cite[(7.4)]{MR735190}. 

There are two questions from \cite{MR735190,MR762896} that still
remain open at this moment. The first, \cite[(5.6)]{MR762896}, concerns
the change of global mean under contraction of an edge. Luo, Xu, Wagner
and Wang \cite{MR4563206} proved the special case when the edge contains
a leaf, while the general case is still open. The other is the caterpillar
conjecture \cite[(7.1)]{MR735190}, which asks whether the tree of
maximal density is a caterpillar. Attempts and progress were made
continuously on this question, with the most recent ones given in
\cite{MR3982896,MR4282633}, but a final answer is still unknown.

This paper mainly concerns the extremal local mean in a tree. The
idea originated from the following question \cite[(7.5)]{MR735190}:
\begin{itemize}
\item For any tree $T,$ is the largest local mean (at order 1) of $T$
always taken on at a leaf?
\end{itemize}
This question itself is solved in \cite{MR3433637} by Wagner and
Wang. They showed that a 1-maximal subtree can have degree equal to
1 or 2. Since local mean is defined for all subtrees and can be compared
among all $k$-subtrees with $k=1,\ldots,n$, it is natural to ask
the following general question about the subtree that attains maximal
local mean among all $k$-subtrees:
\begin{itemize}
\item If $S\subseteq T$ is a $k$-maximal subtree, what can one say about
$S?$
\end{itemize}
The main results are shown in the following two theorems.
\begin{flushleft}
$\mathbf{Result}$ \textbf{1 (see Theorem \ref{thm:mainthm}).} \textit{If
$S$ is a $k$-maximal subtree, then $S$ contains at most one leaf
with $\deg_{T}>2$ and at least one leaf with $\deg_{T}\leq2.$ }
\par\end{flushleft}

\begin{flushleft}
$\mathbf{Result}$\textbf{ 2 (see Theorem \ref{thm:refinement}).}
\textit{Let $S\subseteq T$ be a $k$-maximal subtree. If there exists
a leaf $w\in S$ with $\deg_{T}>2,$ then all other leaves of $S$
have $\deg_{T}=1.$ }
\par\end{flushleft}

To facilitate the proof, we introduce a new notion, $\mathit{index},$
which reflects the exact change of local mean when a subtree includes
a neighbour or excludes a leaf. In brief, when a subtree includes
a new neighbour, its local mean increases by the amount of the index
of that neighbour with respect to the subtree, and when a subtree
excludes a leaf, its local mean decreases by the amount of the index
of that leaf with respect to the subtree. The index not only greatly
simplifies some proofs but also provides some insight into the \foreignlanguage{american}{behaviour}
of local mean and density. Using the index, we give a simple proof
to the following theorem, which was first proved by Wagner and Wang
in \cite{MR3433637}. The analysis of the index also yields a parallel
description of the vertex that attains minimal local mean.
\begin{flushleft}
$\mathbf{Result}$ \textbf{3 (see Theorem \ref{thm:oldthm}).} \textit{The
vertex that attains maximal local mean has $\deg_{T}=1$ or $2.$}
\par\end{flushleft}

\begin{flushleft}
$\mathbf{Result}$ \textbf{4 (see Theorem \ref{thm:minimal_case}).
}\textit{Let $T$ be a tree that is not a path. Then the vertex that
attains minimal local mean has $\deg_{T}\geq3.$}
\par\end{flushleft}

In the last part of this paper, we introduce and study the local density
at a subtree $S\subsetneq T,$ denoted by $D_{T}\left(S\right),$
for the sake of normalizing the local means, so that one can compare
subtrees of $T$ at different orders. Hence it is valid to ask what
is the subtree of $T$ that attains maximal local density. Just as
the density of a tree $T,$ introduced also by Jamison, can be interpreted
as the probability of a random vertex being contained in a random
subtree, the local density at subtree $S$ represents the probability
of a random vertex $\mathit{outside}$ $S$ being contained in a random
subtree $\mathit{containing}$ $S.$ It will be shown that local density
can be arbitrarily close to 1. On the other hand, the lower bound
is characterized in the following theorem.
\begin{flushleft}
$\mathbf{Result}$ \textbf{5 (see Theorem \ref{thm:localdensitylowerbound}).}
Let $S\neq\emptyset$ be a proper subtree of $T.$ Then $D_{T}\left(S\right)\geq\frac{1}{2},$
with equality if and only if $S$ contains the core of $T$ (or equivalently,
$T-S$ contains only paths as components).
\par\end{flushleft}

\section{Notations and Preliminaries}

Let $T=\left(V,E\right)$ be a tree, where $V$ is the set of vertices
and $E$ is the set of edges. For adjacent vertices $v,w\in V,$ denote
the edge by $vw.$ The order of $T,$ denoted by $|T|,$ is the cardinality
of the vertex set of $T.$ For a subtree $S\subseteq T$ and a vertex
$v\in S,$ let $\deg_{S}\left(v\right)$ be the degree of $v$ in
$S,$ and $\deg_{T}\left(v\right)$ be the degree of $v$ in $T.$
If some vertex $w\in T$ lies outside of $S$ and is adjacent to some
vertex in $S,$ then $w$ is called a neighbour of $S$ or equivalently
$w$ is adjacent to $S.$ 

The following are notations related to number and order of subtrees
to be used herein. Let $v\in T-S$ and $(S,v)_{T}$ be the smallest
subtree in $T$ generated by $S$ and $v.$ We define
\begin{align*}
N_{T}: & \quad\textrm{number of subtrees of \ensuremath{T},}\\
R_{T}: & \quad\textrm{total cardinality of all subtrees of \ensuremath{T},}\\
N_{T}\left(S\right): & \quad\textrm{number of subtrees containing \ensuremath{S},}\\
R_{T}\left(S\right): & \quad\textrm{total order of all subtrees containing \ensuremath{S},}\\
N_{T}\left(v;S\right): & \quad\textrm{number of subtrees \ensuremath{R} that contain \ensuremath{v} and satisfy \ensuremath{R\cap(S,v)_{T}=\{v\}}, }\\
R_{T}\left(v;S\right): & \quad\textrm{total cardinality of all subtrees described above}.
\end{align*}
Here $N_{T}\left(v;S\right)$ counts the number of subtrees that are
rooted at $v$ and grow $\mathit{away}$ from $S,$ and $R_{T}\left(v;S\right)$
is the sum of cardinalities of such subtrees. 

Finally, we introduce some notations about mean and density:
\begin{align*}
\mu_{T}\left(S\right)=\frac{R_{T}\left(S\right)}{N_{T}\left(S\right)}: & \quad\textrm{the local mean at subtree \ensuremath{S,}}\\
\mu_{T}=\frac{R_{T}}{N_{T}}: & \quad\textrm{the global mean of \ensuremath{T}, }\\
D_{T}=\frac{\mu_{T}}{|T|}: & \quad\textrm{the global density of \ensuremath{T.}}
\end{align*}

Next, we introduce some preliminary results which will be used throughout
this paper.

Let $v_{1},\ldots,v_{d}$ be the neighbours of $v$ in $T.$ Let $T_{1},\ldots,T_{d}$
be the corresponding subtrees rooted at $v_{1},\ldots,v_{d}$ respectively,
with $v$ removed. Denote $N_{i}=N_{T_{i}}\left(v_{i}\right)$ and
$R_{i}=R_{T_{i}}\left(v_{i}\right),$ for any $i\in\left\{ 1,\ldots,d\right\} .$ 

First, one has
\begin{equation}
N_{T}\left(v\right)=\prod_{i=1}^{d}\left(N_{i}+1\right).\label{eq:2}
\end{equation}
Indeed, to form a subtree containing $v,$ one simply chooses, for
each $i,$ any subtree containing $v_{i}$ from $T_{i}$ (there are
$N_{i}$ ways to complete this step) or leaves it empty (contributing
the $+1$). 

Next, 
\begin{equation}
\mu_{T}\left(v\right)=1+\sum_{i=1}^{d}\frac{R_{i}}{N_{i}+1}.\label{eq:3}
\end{equation}
This equality simply reflects the following fact: the local mean at
$v$ can be calculated by $1,$ which corresponds to the vertex $v,$
plus the adjusted local mean (adjusted to include the empty subtree,
with the +1 in the denominator) at each $v_{i}$ of $T_{i}.$ The
reason for adjusted mean here is that one needs to include the empty
subtree (conventionally the definition of mean subtree order does
not count the empty subtree), as one or more empty subtrees at $v_{i}$
of $T_{i}$ are acceptable when counting subtrees at $v.$ A more
detailed proof can be found in Lemma 3.2 (b) of \cite{MR735190},
or Lemma 1 of \cite{MR762896}. 

From Theorem 3.6 in \cite{MR735190}, one has the following lower
bound for the local mean at $v,$
\begin{equation}
\mu_{T}\left(v\right)\geq\frac{n+1}{2},\label{eq:lowerbound}
\end{equation}
where $n=\left|T\right|$ and equality holds if and only if $T$ is
astral over $v$ ($T$ is either a path or $v$ is the only vertex
of degree $>2$ in $T$).

Now let $v,w$ be two adjacent vertices of $T.$ Denote by $T_{v}$
and $T_{w}$ the two components of $T-vw$ containing $v$ and $w$
respectively, as in Figure~\ref{fig:TvTw}. 

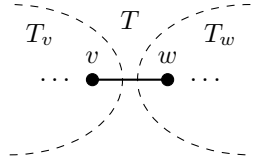
\begin{figure}[h]     
\begin{center}     
\begin{tikzpicture}     
	{          
	\draw[thick] (0,0)--(1,0);    
	\draw[fill] (0,0) circle (0.08);     
	\draw[fill] (1,0) circle (0.08);
	
	\node at (-0.5,0) {$\ldots$};
	
	\draw[dashed] (2.1,1) arc (90:270:1.5cm and 1cm);
	\draw[dashed] (-1.1,-1) arc (-90:90:1.5cm and 1cm);
	
	\node at (0.5,0.8) {$T$};
	\node at (1.7,0.6) {$T_w$};
	\node at (-0.7,0.6) {$T_v$};
    \node at (0,0.3) {$v$};
	\node at (1,0.3) {$w$};
	\node at (1.5,0) {$\ldots$};    
	     } 	
\end{tikzpicture}\\ 
\end{center} 
\caption{$T_v$ and $T_w$.}
\label{fig:TvTw} 
\end{figure}

Then
\begin{equation}
R_{T}\left(v,w\right)=N_{T_{w}}(w)R_{T_{v}}(v)+R_{T_{w}}(w)N_{T_{v}}(v),\label{eq:4}
\end{equation}
where $R_{T}\left(v,w\right)$ is the total order of subtrees of $T$
containing both $v$ and $w.$ We split every subtree containing $v$
and $w$ into two parts in a natural way, namely the subtree rooted
at $v$ without $w$ and the subtree rooted at $w$ without $v,$
and then count. Now each subtree at $v$ combines with each and every
subtree at $w$ exactly once in the sum of $R_{T}(v,w),$ hence it
occurs exactly as many times as the number of subtrees at $w$ of
$T_{w},$ i.e., $N_{T_{w}}(w)$ times. So each subtree at $v$ contributes
to $R_{T}\left(v,w\right)$ with a factor of $N_{T_{w}}(w),$ summing
over all subtrees at $v$ gives $N_{T_{w}}(w)R_{T_{v}}(v),$ which
is the first summand of (\ref{eq:4}). A symmetric argument gives
the second summand.

\medskip{}

The last observations are concerned with contraction with respect
to a subtree. 

For a tree $T$ and a subtree $S,$ let $T/S$ be the new tree obtained
by collapsing $S$ to a single vertex $s$ and keeping all the adjacency
relations in the natural way: a vertex is adjacent to $S$ in $T$
if and only if it is adjacent to $s$ in $T/S.$

Now let $S\subseteq T$ be a subtree and $U\subseteq S$ a subtree
of $S.$ Then one has a natural bijection induced by contraction with
respect to $U:$
\[
\left\{ \textrm{subtrees of \ensuremath{T} containing \ensuremath{S}}\right\} \leftrightarrow\left\{ \textrm{subtrees of \ensuremath{T/U} containing \ensuremath{S/U}}\right\} .
\]
Furthermore, through the same bijection, one sees that each subtree
of $T$ in the left set has an order greater than the corresponding
subtree of $T/U$ in the right set by exactly $|U|-1.$ Hence, the
same applies to the local means, namely
\begin{equation}
\mu_{T}\left(S\right)=\mu_{T/U}\left(S/U\right)+|U|-1.\label{eq:5}
\end{equation}

\section{On maximal local means}
\begin{lem}[Index Lemma]
\label{lem:lem3.1}If $w$ is a neighbour of subtree $S\subsetneq T,$
then 
\[
\mu_{T}\left(S+w\right)=\mu_{T}\left(S\right)+\frac{R_{T}\left(w;S\right)}{N_{T}\left(w;S\right)\left(1+N_{T}\left(w;S\right)\right)}.
\]
\end{lem}

The last expression on the right of the above equation turns out to
appear in many different places. We give it a name for the sake of
future convenience.
\begin{defn*}
Let $T_{v}$ be a rooted tree with root $v.$ Define the $\mathit{index}$
of the rooted tree $T_{v}$ to be
\[
i\left(T_{v}\right)\coloneqq\frac{R_{T_{v}}\left(v\right)}{N_{T_{v}}\left(v\right)\left(1+N_{T_{v}}\left(v\right)\right)}.
\]
Furthermore, let $T$ be a tree and $S\subseteq T$ a subtree. For
any vertex $v$ outside of $S$ or a leaf of $S,$ define the $\mathit{index}$
of $v$ with respect to $S$ to be
\begin{equation}
i\left(v;S\right):=\frac{R_{T}\left(v;S\right)}{N_{T}\left(v;S\right)\left(1+N_{T}\left(v;S\right)\right)}.
\end{equation}
\end{defn*}
\begin{rem*}
With notations as above, $i\left(v;S\right)$ is the same as $i\left(T_{v}\right),$
where $T_{v}$ is the rooted tree at $v$ that ``grows away'' from
$S.$

The index of a vertex $v$ is well-defined once a subtree rooted at
$v$ is assigned by the context without ambiguity. When this reference
of ``direction'' is clear from context, we will sometimes drop the
second parameter and write $i\left(v\right).$
\end{rem*}
\begin{example*}
In Figure~\ref{fig:exIndex}, $T_{w}$ (indicated by the dashed ellipse),
is the rooted tree at $w$ that grows away from $S.$ We have
\[
N\left(w;S\right)=4,R\left(w;S\right)=8,N\left(v;S\right)=5,R\left(v;S\right)=13,
\]
and 
\[
i\left(w;S\right)=i\left(T_{w}\right)=\frac{2}{5},i\left(v;S\right)=\frac{13}{30}.
\]
\hfill \qedsymbol
\end{example*}
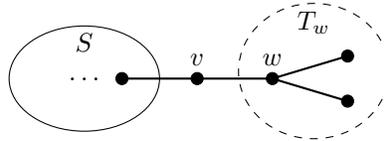
\begin{figure}[h]     
\begin{center}     
\begin{tikzpicture}     
	{          
	\draw[thick] (0,0)--(1,0) -- (2,0);
	\draw[thick] (3,0.3)--(2,0) -- (3,-0.3);    
	\draw[fill] (0,0) circle (0.08);     
	\draw[fill] (1,0) circle (0.08);
	\draw[fill] (2,0) circle (0.08);     
	\draw[fill] (3,0.3) circle (0.08);
	\draw[fill] (3,-0.3) circle (0.08); 
	\draw[] (-0.5,0) ellipse (1cm and 0.7cm); 
	\draw[dashed] (2.55,0.1) ellipse (1cm and 0.9cm); 
	\node at (-0.5,0) {$\ldots$};
	\node at (-0.5,0.48) {$S$};
	\node at (2.55,0.75) {$T_w$};
    
	\node at (1,0.25) {$v$};     
	\node at (2,0.25) {$w$};     } 	
\end{tikzpicture}\\ 
\end{center} 
\caption{Indices of {$v$} and {$w$}.}
\label{fig:exIndex} 
\end{figure} 

Let $w$ be a neighbour of $v.$ It is proved in \cite[Lemma 3.2c]{MR735190}
that 
\begin{equation}
i\left(v;w\right)\leq\frac{1}{2}\label{eq:0}
\end{equation}
with equality if and only if the component in $T-w$ containing $v$
is a leaf or a path, as in Figure~\ref{fig:maxindex}. Following the
conventions in \cite{MR3982896}, we call a maximal path containing
a leaf and vertices of degree 2 a $limb$ and call a vertex $v$ in
a limb a $\mathit{limb}$ $\mathit{vertex}.$ Note that $v$ is a
limb vertex if and only if $i\left(v;w\right)=\frac{1}{2}$ for $\mathit{some}$
neighbour $w$ of $v.$ Moreover, call the remaining part of a tree,
namely $T-\left\{ \textrm{limbs}\right\} ,$ the $core$ of $T,$
denoted by $T^{*}.$ 

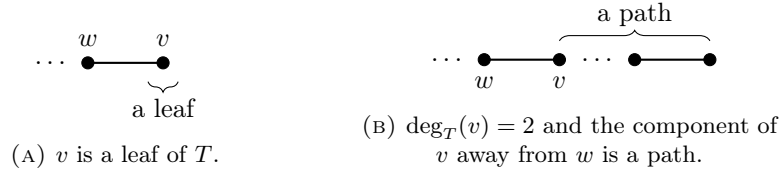
\begin{figure}[h]
\captionsetup{justification=centering}
\begin{minipage}[b]{.45\linewidth} 
\begin{center}     
\begin{tikzpicture}     
	{          
	
	\draw[thick] (-0.5,0)--(0.5,0);

	\draw[fill] (-0.5,0) circle (0.08);     
	\draw[fill] (0.5,0) circle (0.08);

	\draw [decorate,decoration={brace,amplitude=4pt},xshift=0pt,yshift=0pt] (0.7,-0.25) -- (0.3,-0.25) node [black,midway,yshift=-0.35cm]{a leaf};
	
	\node at (-1,0) {$\ldots$};
	\node at (0.5,0.3) {$v$};     
	\node at (-0.5,0.3) {$w$};
   	} 	
\end{tikzpicture}\\ 
\subcaption{$v$ is a leaf of $T$.} 
\end{center} 
\end{minipage}\quad\begin{minipage}[b]{.45\linewidth}
\begin{center}     
\begin{tikzpicture}
{ 	
	\draw[thick] (3,0)--(4,0);
	\draw[thick] (5,0) -- (6,0);
   
	\draw[fill] (3,0) circle (0.08);     
	\draw[fill] (4,0) circle (0.08);
	 
	\draw[fill] (5,0) circle (0.08); 
	\draw[fill] (6,0) circle (0.08);

	\draw [decorate,decoration={brace,amplitude=4pt},xshift=0pt,yshift=0pt] (4,0.25) -- (6,0.25) node [black,midway,yshift=0.35cm]{a path};
	
    \node at (4.5,0) {$\ldots$};
	\node at (2.5,0) {$\ldots$};
	\node at (4,-0.3) {$v$};     
	\node at (3,-0.3) {$w$};
	
	} 	
\end{tikzpicture}\\ 
\subcaption{$\text{deg}_T(v)=2$ and the component of $v$ away from $w$ is a path.} 
\end{center} 
\end{minipage} 

\caption{$i(v;w)=\frac{1}{2}$.}\label{fig:maxindex} \end{figure} 
\begin{rem*}
Being a limb vertex is a necessary but not sufficient condition for
having index $\frac{1}{2},$ as index depends on the reference to
a subtree. It is possible for some limb vertex $v$ to have an index
not equal to $\frac{1}{2}.$ For a simple example, consider the tree
in Figure~\ref{fig:leg}. Then $v$ is a limb vertex (in fact, so
are all vertices except for $c$), as $i\left(v;S\right)=i\left(v;c\right)=\frac{1}{2},$
but $i\left(v;w\right)\neq\frac{1}{2}.$ 
\end{rem*}
\begin{figure}[h]     
\begin{center}     
\begin{tikzpicture}     
	{          
	\draw[thick] (0,0)--(1,0) -- (2,0);
	\draw[thick] (3,0.6)--(2,0) -- (3,-0.6);    
	\draw[fill] (0,0) circle (0.08);     
	\draw[fill] (1,0) circle (0.08);
	\draw[fill] (2,0) circle (0.08);     
	\draw[fill] (3,0.6) circle (0.08);
	\draw[fill] (3,-0.6) circle (0.08); 
	\draw[] (2.6,0) ellipse (1cm and 1.2cm); 
	
	\node at (2.5,0.8) {$S$};
    \node at (2,0.3) {$c$};
	\node at (1,0.3) {$v$};     
	\node at (0,0.3) {$w$};     } 	
\end{tikzpicture}\\ 
\end{center} 
\caption{Limb vertex.}
\label{fig:leg} 
\end{figure}
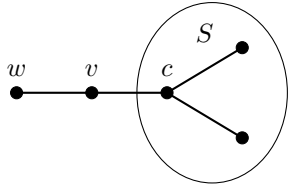 

We now restate Lemma \ref{lem:lem3.1} using the index and prove it.
\begin{lem*}[Index Lemma restated]
 If $w$ is a neighbour of subtree $S\subsetneq T,$ then 
\[
\mu_{T}\left(S+w\right)=\mu_{T}\left(S\right)+i\left(w;S\right).
\]
 Similarly, if $v$ is a leaf of $S\subseteq T$ and $S\neq\{v\},$
then 
\[
\mu_{T}\left(S-v\right)=\mu_{T}\left(S\right)-i\left(v;S\right).
\]
\end{lem*}
\begin{proof}[Proof of the Index Lemma]
By (\ref{eq:5}), one can always contract $S$ and it suffices to
show the case when $S$ is a single vertex, say $v.$ So the situation
reduces to that for adjacent vertices $v,w\in T.$ Let $T_{v},T_{w}$
be the \foreignlanguage{american}{components} of $T-vw$ that contain
$v,w$ respectively. We need to show 
\begin{align*}
\mu_{T}\left(v,w\right)-\mu_{T}\left(v\right) & =i\left(w;v\right)\\
\frac{R_{T}\left(v,w\right)}{N_{T}\left(v,w\right)}-\frac{R_{T}\left(v\right)}{N_{T}\left(v\right)} & =i\left(w;v\right).
\end{align*}
Note that any subtree of $T$ containing $v$ is either a subtree
of $T_{v}$ containing $v$ or a subtree of $T$ containing both $v$
and $w.$ Hence, $N_{T}\left(v\right)=N_{T_{v}}\left(v\right)\left(N_{T_{w}}\left(w\right)+1\right)$
and $R_{T}\left(v\right)=R_{T_{v}}\left(v\right)+R_{T}\left(v,w\right).$
Then using (\ref{eq:4}), one has,
\begin{align*}
\frac{R_{T}\left(v,w\right)}{N_{T}\left(v,w\right)}-\frac{R_{T}\left(v\right)}{N_{T}\left(v\right)} & =\frac{N_{T_{v}}\left(v\right)R_{T_{w}}\left(w\right)+R_{T_{v}}\left(v\right)N_{T_{w}}\left(w\right)}{N_{T_{v}}\left(v\right)N_{T_{w}}\left(w\right)}\\
 & \quad-\frac{R_{T_{v}}\left(v\right)+\left(N_{T_{v}}\left(v\right)R_{T_{w}}\left(w\right)+R_{T_{v}}\left(v\right)N_{T_{w}}\left(w\right)\right)}{N_{T_{v}}\left(v\right)\left(N_{T_{w}}\left(w\right)+1\right)}\\
 & =\frac{N_{T_{v}}\left(v\right)R_{T_{w}}\left(w\right)}{N_{T_{v}}\left(v\right)N_{T_{w}}\left(w\right)\left(N_{T_{w}}\left(w\right)+1\right)}\\
 & =\frac{R_{T_{w}}\left(w\right)}{N_{T_{w}}\left(w\right)\left(N_{T_{w}}\left(w\right)+1\right)}\\
 & =i\left(w;v\right).
\end{align*}
\end{proof}
\begin{rem*}
Since the index is a positive number whose value lies in $(0,\frac{1}{2}],$
the previous lemma immediately implies 
\[
\mu_{T}\left(v\right)<\mu_{T}\left(v,w\right)\leq\mu_{T}\left(v\right)+\frac{1}{2},
\]
for any adjacent vertices $v,w$ in $T$ (\cite[Lemma 4.1]{MR735190}),
and other 'continuity' results \cite[Theorem 4.3 and 4.5]{MR735190}.
\end{rem*}
In \cite{MR3433637}, it is shown that a vertex achieving maximal
local mean has degree 1 or 2. An immediate result following the Index
Lemma is a refinement of this description.

Let $T$ be a tree that is not a path. Consider the set of all maximal
paths in $T,$ namely all subtrees that are paths whose end-vertices
both have degree $\neq2$ and all other vertices have degree 2 in
$T.$ Observe that there are only two kinds of paths in this set,
namely those paths with both end-vertices of degree $\ge3,$ which
will be called $core$-$paths$ (as they appear only in the core of
a tree), and limbs. Then the refinement can be stated as follows.
\begin{thm}
\label{thm:branchingpath}Let $T$ be a tree that is not a path. Let
$v\in T$ be a vertex that achieves the maximal local mean at size
1. Then $v$ is either a leaf, or it lies in a core-path.
\end{thm}

\begin{proof}
We only need to eliminate the possibility of $v$ being a limb vertex.
Suppose $v$ lies in a limb and $v$ is not a leaf. Denote the neighbour
of $v$ closer to the leaf by $u.$ Observe that the component containing
$u$ in $T-v$ is a path, while the component containing $v$ in $T-u$
is not, due to the existence of the branching vertex at the other
end of the leaf path. So we have $i\left(u;v\right)=1/2$ and $i\left(v;u\right)<1/2.$
Then by the Index Lemma, $\mu(u)=\mu\left(v,u\right)-i\left(v;u\right)>\mu\left(v,u\right)-i\left(u;v\right)=\mu(v).$
This contradicts the assumption that $v$ achieves maximal local mean.
\end{proof}
Next, we state the aforementioned result regarding the $k$-subtree
that achieves maximal local mean.
\begin{thm}
\label{thm:mainthm}Let $S\subseteq T$ with $|S|=k$ where $k$ is
some positive integer. If $S$ is a $k$-maximal subtree, then $S$
contains at most one leaf with $\deg_{T}>2$ and at least one leaf
with $\deg_{T}\leq2.$ 

In particular, when $k=1,$ the vertex that achieves the maximum of
the local mean has degree 1 or 2.
\end{thm}

\begin{proof}
In the case $k=1,$ the statement is first proved in \cite[Theorem 3.2]{MR3433637}
and an alternative proof will be given in Theorem \ref{thm:oldthm}.
Assume $k>1.$ To start, fix a neighbour $w$ of $S$ that has the
minimal $N_{T}\left(w;S\right),$ i.e., $N_{T}\left(w;S\right)\leq N_{T}\left(a;S\right)$
for any neighbour $a$ of $S.$ Note that the vertex adjacent to $w$
inside $S$ could be a leaf of $S$ or not. Next we pick a leaf of
$S$ with $\deg_{T}>2$ and not adjacent to $w.$ If there exists
no such leaf in $S,$ then the statement is automatically true. Now
assume $v$ is such a leaf of $S$ with $\deg_{T}>2.$ Then $i\left(v;S\right)<\frac{1}{2}$
by (\ref{eq:0}) as it cannot be a limb vertex. The strategy is to
show that including $w$ and removing $v$ increases the local mean
and then reach a contradiction. By the Index Lemma, this is equivalent
to $i\left(w\right)>i\left(v\right),$ i.e.,
\begin{equation}
\frac{R\left(w\right)}{N\left(w\right)\left(1+N\left(w\right)\right)}>\frac{R\left(v\right)}{N\left(v\right)\left(1+N\left(v\right)\right)}.\label{eq:7}
\end{equation}
(The subscript $T$ is dropped as there is no risk of ambiguity.)
Now, for the purpose of establishing an inequality, we exclude two
cases: either $w$ is itself a leaf in $T,$ or the only neighbour
of $w$ outside of $S$ is a leaf in $T.$ In both cases, $i\left(w\right)=\frac{1}{2}>i\left(v\right)$
and the statement follows. Now with these two cases excluded, one
has
\[
R\left(w\right)\geq2N\left(w\right).
\]
Indeed, in all cases other than the two above, there is at least one
subtree rooted at $w$ (outside of $S$) having order 3, and all other
such subtrees, except the singleton subtree $\left\{ w\right\} ,$
have order at least 2. Hence, 
\[
R\left(w\right)\geq1+3+2\left(N\left(w\right)-2\right)=2N\left(w\right).
\]
So the left hand side of (\ref{eq:7}) satisfies
\[
\frac{R\left(w\right)}{N\left(w\right)\left(1+N\left(w\right)\right)}\geq\frac{2}{1+N\left(w\right)}.
\]
Denote all the outgoing (with respect to $S$) \foreignlanguage{american}{neighbours}
of $v$ by $v_{1},\ldots,v_{l},$ and let $N_{i}:=N\left(v_{i};v\right),R_{i}:=R\left(v_{i};v\right).$
From the choice of $w,$ we know that $N\left(w\right)\leq N_{i}$
for all $i\in\left\{ 1,\ldots,l\right\} .$ Continuing with the right
hand side of (\ref{eq:7}), we have
\begin{align*}
\frac{R\left(v\right)}{N\left(v\right)\left(1+N\left(v\right)\right)} & =\left(1+\sum_{i=1}^{l}\frac{R_{i}}{1+N_{i}}\right)\frac{1}{1+\prod_{i=1}^{l}\left(N_{i}+1\right)}\\
 & \leq\frac{1+\sum_{i=1}^{l}\frac{N_{i}}{2}}{1+\prod_{i=1}^{l}\left(N_{i}+1\right)},
\end{align*}
where the first step used equations (\ref{eq:2}) and (\ref{eq:3}),
and the second step used (\ref{eq:0}). What is left to show is 
\begin{align}
\frac{2}{1+N\left(w\right)} & >\frac{1+\sum_{i=1}^{l}\frac{N_{i}}{2}}{1+\prod_{i=1}^{l}\left(N_{i}+1\right)}\nonumber \\
2 & >\frac{\left(1+\sum_{i=1}^{l}\frac{N_{i}}{2}\right)\left(1+N\left(w\right)\right)}{1+\prod_{i=1}^{l}\left(N_{i}+1\right)},\label{eq: eqinproof3.3}
\end{align}
where $1\leq N\left(w\right)\leq N_{i},$ $l\geq2$ (this follows
from the assumption $\deg_{T}>2$). For simplicity, denote $N:=N\left(w\right)$
for the rest of the proof.

The right hand side of (\ref{eq: eqinproof3.3}) gives
\[
\frac{\left(1+\sum_{i=1}^{l}\frac{N_{i}}{2}\right)\left(1+N\right)}{1+\prod_{i=1}^{l}\left(N_{i}+1\right)}=\frac{1+\frac{1}{2}\sum_{i=1}^{l}N_{i}+N+\frac{1}{2}\sum_{i=1}^{l}N_{i}N}{2+\sum_{i=1}^{l}N_{i}+\sum_{i,j=1,i<j}^{l}N_{i}N_{j}+\left(\textrm{higher order terms}\right)}.
\]
Since $N\leq N_{i}$ for all $i,$ if one collects and compares terms
from numerator and denominator by their powers, one has
\begin{align*}
\textrm{linear terms:} & \;\frac{1}{2}\sum_{i=1}^{l}N_{i}+N\leq\sum_{i=1}^{l}N_{i},\\
\textrm{quadratic terms:} & \;\frac{1}{2}\sum_{i=1}^{l}N_{i}N\leq\sum_{i=1,i<j}^{l}N_{i}N_{j},
\end{align*}
for all $l\geq2$ and positive integers $N_{i}$ and $N$ such that
$N\le N_{i}$ for all $i.$ Hence the right hand side of (\ref{eq: eqinproof3.3})
is actually less than 1 thus trivially less than 2, and the statement
follows.
\end{proof}
\begin{rem*}
By the above theorem, one of the following two descriptions regarding
leaves of a $k$-maximal subtree $S\subseteq T$ is true:
\end{rem*}
\begin{enumerate}
\item[case I:] all leaves of $S$ have $\deg_{T}\leq2,$ 
\item[case II:] there is only one leaf of $S$ with $\deg_{T}>2$ and one of its
neighbours outside of $S,$ say $w,$ satisfies $N_{T}\left(w;S\right)\leq N_{T}\left(a;S\right)$
for every neighbour $a$ of $S.$ 
\end{enumerate}
\hfill \qedsymbol

The next example shows that both situations can occur.
\begin{example*}
Let $T$ be the tree obtained by connecting the \foreignlanguage{british}{centres}
of two 2-stars $K_{1,2}$ by a path of $2n$ vertices $P_{2n}.$ Let
$v,w$ be the two adjacent vertices in the middle of the path, as
shown in Figure~\ref{fig:bothcase}. 
\end{example*}
\begin{figure}[h]
\captionsetup{justification=centering}
\begin{center}     
\begin{tikzpicture}     
	{ 	
	\draw[thick] (1,0)--(2,0);
	\node at (0.5,0) {$\ldots$};
	\node at (2.5,0) {$\ldots$};
	\node at (1,-0.3) {$v$};
	\node at (2,-0.3) {$w$};
	
	\draw [decorate,decoration={brace,amplitude=4pt},xshift=0pt,yshift=0pt] (0,0.3) -- (1,0.3) node [black,midway,yshift=0.35cm]{nodes};
	\node at (0.5,0.95) {$n+1$};

	\draw [decorate,decoration={brace,amplitude=4pt},xshift=0pt,yshift=0pt] (2,0.3) -- (3,0.3) node [black,midway,yshift=0.35cm]{nodes};
	\node at (2.5,0.95) {$n+1$};

	\draw[thick] (4,0.6)--(3,0) -- (4,-0.6);    
	
	\draw[fill] (0,0) circle (0.08);     
	\draw[fill] (1,0) circle (0.08);
	\draw[fill] (2,0) circle (0.08);
	\draw[fill] (3,0) circle (0.08);
	\node at (3,-0.3) {$c$};
	\node at (4.3,-0.6) {$d$};
	     
	\draw[fill] (4,0.6) circle (0.08);
	\draw[fill] (4,-0.6) circle (0.08); 

	\draw[thick] (-1,0.6)--(0,0) -- (-1,-0.6); 
     
	\draw[fill] (-1,0.6) circle (0.08);

	\draw[fill] (-1,-0.6) circle (0.08);

	} 	
	\end{tikzpicture}\\

\end{center} 
\caption{$n=1$: \{$c,d$\} is 2-maximal; $n>1$: \{$v,w$\} is 2-maximal. }\label{fig:bothcase} \end{figure}
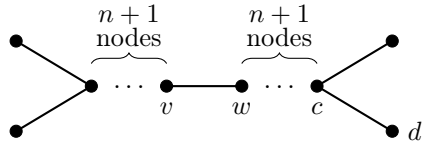

First, we have $\mu_{T}\left(v,w\right)\geq\mu_{T}\left(a,b\right)$
for any adjacent pair $a,b$ from the path. (The proof is postponed
till Section 5 (Figure~\ref{fig:twostars}), where we consider a more
general setup.) Moreover, by the bijection induced by the contraction
and equation (\ref{eq:5}), the calculation in \cite{MR3433637} implies
that $\{v,w\}$ has a local mean (at order 2) of
\[
\mu_{T}(v,w)=1+\left(n+2\right)\left(n+6\right)/\left(n+4\right).
\]
In order to determine which 2-subtree achieves the maximum of the
local mean at order 2 in $T,$ we need to calculate $\mu_{T}(c,d).$
We have

\begin{align*}
\mu_{T}\left(c,d\right) & =1+\mu_{T-d}\left(c\right)\\
 & =1+1+\frac{1}{2}+\frac{\left(1+2\cdot2+3\right)+2n\cdot4+\left(1+2+\cdots+2n\right)}{2n+4+1}\\
 & =\frac{4n^{2}+28n+41}{4n+10},
\end{align*}
where the second step follows from (\ref{eq:3}). Taking the difference,
\[
\mu_{T}\left(v,w\right)-\mu_{T}\left(c,d\right)=\frac{2n^{2}+n-4}{\left(4n+10\right)\left(n+4\right)}.
\]
To conclude, when $n=1,$ the subtree $\{c,d\},$ which has one leaf
of degree 1 and the other of degree 3, is 2-maximal in $T;$ when
$n>1,$ the subtree $\{v,w\},$ whose both leaves are of degree 2,
is 2-maximal in $T.$ \hfill \qedsymbol

Now, a natural question to ask is: 
\begin{quote}
When does a $k$-maximal subtree $S$ contain a leaf with $\deg_{T}>2?$ 
\end{quote}
The next theorem tells us that a necessary condition for it to occur
is that any other leaf of $S$ is actually a leaf of $T.$
\begin{thm}
\label{thm:refinement}Let $S\subseteq T$ be a $k$-maximal subtree.
If there exists a leaf $w\in S$ with $\deg_{T}>2,$ then all other
leaves of $S$ have $\deg_{T}=1.$ 
\end{thm}

\begin{proof}
Let $w_{1},\ldots,w_{d}$ be the \foreignlanguage{british}{neighbours}
of $w$ outside of $S.$ From the assumption, we know $d\geq2.$ Denote
the components in $T-S$ containing $w_{j}$ by $T_{j}.$ Assume,
for the purpose of contradiction, that $v\neq w$ is a leaf of $S$
with $\deg_{T}>1.$ By Theorem \ref{thm:mainthm}, we know $\deg_{T}\left(v\right)=2.$
Denote the only neighbour of $v$ outside of $S$ by $v_{1},$ and
its corresponding component in $T-S$ by $C_{1},$ as shown in Figure~\ref{fig:refineproof}.
From the remark right after the proof of Theorem \ref{thm:mainthm},
we know that one of $w_{1},\ldots,w_{d}$ has the least number of
rooted subtrees outside of $S$ among all neighbours of $S.$ Without
loss of generality, let it be $w_{1}.$ Then we have $N\left(w_{1};S\right)\leq N\left(w_{j};S\right),j=2,\ldots,d,$
and $N\left(w_{1};S\right)\leq N\left(v_{1};S\right).$ For simplicity,
we drop the reference to $S$ in this proof.

\begin{figure}[h]
\captionsetup{justification=centering}
\begin{center}     
\begin{tikzpicture}     
	{ 	
	
	\node at (0,-0.1) {$\ldots$};
	\draw[thick] (-0.6,0.5)--(-0.3,0); 
	\draw[thick] (0.6,0.5)--(0.3,0);

	\node at (-0.8,0.35) {$v$};
	\node at (0.85,0.35) {$w$};   
	
	\draw[fill] (-0.6,0.5) circle (0.07);     
	\draw[fill] (0.6,0.5) circle (0.07);
	\draw[fill] (-0.9,1.2) circle (0.07);
	\draw[fill] (0.3,1.2) circle (0.07);
	\draw[fill] (1.2,1.2) circle (0.07);

	\node at (-0.9,1.5) {$v_1$};
	\node at (0.3,1.5) {$w_1$};
	\node at (1.2,1.5) {$w_d$};

	\node at (-0.9,2) {$\vdots$};
	\node at (0.3,2) {$\vdots$};
	\node at (1.2,2) {$\vdots$};

	\node at (-0.9,2.3) {$C_1$};
	\node at (0.3,2.3) {$T_1$};
	\node at (1.2,2.3) {$T_d$};

	\node at (0.75,1.15) {$\ldots$};

	\draw[thick,dashed] (0.6,0.5)--(1.2,1.2); 
	\draw[thick,dashed] (0.6,0.5)--(0.3,1.2);
	\draw[thick,dashed] (-0.6,0.5)--(-0.9,1.2);

	\draw[dashed,anchor=center] (-1.28,2.3) arc (180:360:0.4cm and 1.3cm);
	\draw[dashed,anchor=center] (-0.05,2.3) arc (180:360:0.35cm and 1.3cm);
	\draw[dashed,anchor=center] (0.87,2.3) arc (180:360:0.35cm and 1.3cm);

	\draw[dashed,anchor=center] (-1.45,0) arc (180:0:1.5cm and 0.8cm);

	\node at (0,0.3) {$S$};
	
	} 	
	\end{tikzpicture}\\

\end{center} 
\caption{$\text{deg}_T(v)=2$ and $\text{deg}_T(w)>2$.}\label{fig:refineproof} \end{figure}

Since $S$ is $k$-maximal, from the Index Lemma, we know $i\left(v\right)\geq i\left(w_{j}\right),j=1,\ldots,d.$
Equivalently,
\[
\frac{R\left(w_{j}\right)}{N\left(w_{j}\right)+1}\leq i\left(v\right)N\left(w_{j}\right).
\]
Assume that $v$ is a limb vertex, i.e. $i\left(v\right)=\frac{1}{2},$
which implies $i\left(v_{1}\right)=\frac{1}{2}.$ With $d\geq2,$
$w$ cannot be a limb vertex, so $i\left(w\right)<\frac{1}{2}.$ Hence,
$i\left(v_{1}\right)>i\left(w\right),$ which means one can obtain
a higher local mean of order $k$ by removing $w$ and including $v_{1},$
contradicting that $S$ is $k$-maximal. So, $i\left(v\right)<\frac{1}{2}$
and it follows that $i\left(w_{j}\right)<\frac{1}{2}$ for all $j=1,\ldots,d.$
Then $N\left(w_{j}\right)\geq4$ for all $j=1,\ldots,d.$ Indeed,
as $w_{j}$ is not a limb vertex, the smallest case for $T_{j}$ is
the 3-star $K_{2,1}$ with $w_{j}$ at the centre, and the number
of subtrees containing the centre is 4.

What is left to show is that $i\left(v_{1}\right)>i\left(w\right).$
We start with the observations that $N\left(v\right)=N\left(v_{1}\right)+1$
and $R\left(v\right)=R\left(v_{1}\right)+N\left(v_{1}\right)+1.$
So we have
\begin{align*}
i\left(v_{1}\right) & =\frac{R\left(v_{1}\right)}{N\left(v_{1}\right)\left(N\left(v_{1}\right)+1\right)}\\
i\left(v_{1}\right) & =\frac{R\left(v\right)-N\left(v\right)}{\left(N\left(v\right)-1\right)N\left(v\right)}\\
i\left(v_{1}\right) & =\frac{R\left(v\right)}{\left(N\left(v\right)-1\right)N\left(v\right)}-\frac{1}{N\left(v\right)-1}\\
i\left(v_{1}\right) & >i\left(v\right)-\frac{1}{N\left(v_{1}\right)}\\
i\left(v_{1}\right)+\frac{1}{N\left(v_{1}\right)} & >i\left(v\right)\\
2\,i\left(v_{1}\right) & >i\left(v\right).
\end{align*}
Indeed, the last inequality follows from $i\left(v_{1}\right)>\frac{1}{N\left(v_{1}\right)},$
which is further equivalent to $\mu_{C_{1}}\left(v_{1}\right)>\frac{N\left(v_{1}\right)+1}{N\left(v_{1}\right)}.$
The last inequality holds trivially given that $N\left(v_{1}\right)\geq4.$
Indeed, $N\left(v_{1}\right)\geq4$ implies that $C_{1}$ has at least
3 vertices, then $\mu_{C_{1}}\left(v_{1}\right)\ge\frac{3+1}{2}>\frac{5}{4}\ge\frac{N\left(v_{1}\right)+1}{N\left(v_{1}\right)}.$
On the other hand,
\begin{align*}
i\left(w\right) & =\frac{R\left(w\right)}{N\left(w\right)\left(N\left(w\right)+1\right)}\\
 & =\frac{1+\sum_{j=1}^{d}\frac{R\left(w_{j}\right)}{N\left(w_{j}\right)+1}}{\prod_{j=1}^{d}\left(N\left(w_{j}\right)+1\right)+1}\\
 & =\frac{1+\sum_{j=1}^{d}i\left(w_{j}\right)N\left(w_{j}\right)}{\prod_{j=1}^{d}\left(N\left(w_{j}\right)+1\right)+1}\\
 & \leq\frac{1+\sum_{j=1}^{d}i\left(v\right)N\left(w_{j}\right)}{\prod_{j=1}^{d}\left(N\left(w_{j}\right)+1\right)+1}\\
 & =\frac{1+i\left(v\right)\sum_{j=1}^{d}N\left(w_{j}\right)}{\prod_{j=1}^{d}\left(N\left(w_{j}\right)+1\right)+1}\\
 & <\frac{1+2\,i\left(v_{1}\right)\sum_{j=1}^{d}N\left(w_{j}\right)}{\prod_{j=1}^{d}\left(N\left(w_{j}\right)+1\right)+1}.
\end{align*}
Next, we claim that $1<i\left(v_{1}\right)\sum_{j=1}^{d}N\left(w_{j}\right)$
and postpone the proof to the end. With this inequality, one proceeds
as follows:
\begin{align*}
i\left(w\right) & <\frac{3\,i\left(v_{1}\right)\sum_{j=1}^{d}N\left(w_{j}\right)}{\prod_{j=1}^{d}\left(N\left(w_{j}\right)+1\right)+1}\\
 & =\frac{3\sum_{j=1}^{d}N\left(w_{j}\right)}{\prod_{j=1}^{d}\left(N\left(w_{j}\right)+1\right)+1}i\left(v_{1}\right)\\
 & <i\left(v_{1}\right).
\end{align*}
To see the last step, recall that $N\left(w_{j}\right)\geq4$ for
all $j$ and $d\geq2.$ For each fixed $d,$ the coefficient of $i\left(v_{1}\right)$
on the right hand side is a decreasing function with respect to each
$N(w_{j}).$ Indeed, if we temporarily denote $N(w_{j})$ by $x_{j}$
and the entire coefficient part of $i(v_{1})$ by $y,$ then taking
the derivative of $y$ with respect to $x_{i}$ leads to
\begin{align*}
\frac{\partial y}{\partial x_{i}} & =\frac{\partial}{\partial x_{i}}\left(\frac{3\sum_{j=1}^{d}x_{j}}{\prod_{j=1}^{d}\left(x_{j}+1\right)+1}\right)\\
 & =\frac{3\left(\prod_{j\neq i}\left(x_{j}+1\right)+1\right)-3\sum_{j\neq i}x_{j}\prod_{j\neq i}\left(x_{j}+1\right)}{\left(\prod_{j=1}^{d}\left(x_{j}+1\right)+1\right)^{2}}\\
 & =\frac{3\prod_{j\neq i}\left(x_{j}+1\right)\left(1-\sum_{j\neq i}x_{j}\right)+3}{\left(\prod_{j=1}^{d}\left(x_{j}+1\right)+1\right)^{2}},
\end{align*}
which is negative if $x_{j}\geq4$ for all $j.$ Therefore, it suffices
to show that $y$ is less than 1 for every $d\geq2$ when $x_{j}=4$
for all $j.$ But in this case, the coefficient $y$ reduces to $\frac{12d}{5^{d}+1},$
which is obviously less than 1 for all $d\geq2.$ 

At last, we prove the claim $1<i\left(v_{1}\right)\sum_{j=1}^{d}N\left(w_{j}\right).$
We know that for all $j=1,\ldots,d,$
\begin{align*}
i\left(v\right) & \geq i\left(w_{j}\right)\\
\frac{\mu\left(v\right)}{N\left(v\right)+1} & \geq\frac{\mu_{T_{j}}\left(w_{j}\right)}{N\left(w_{j}\right)+1}\\
\frac{1+\frac{R\left(v_{1}\right)}{N\left(v_{1}\right)+1}}{N\left(v_{1}\right)+2} & \geq\frac{\mu_{T_{j}}\left(w_{j}\right)}{N\left(w_{j}\right)+1}\\
\frac{R\left(v_{1}\right)}{\left(N\left(v_{1}\right)+1\right)\left(N\left(v_{1}\right)+2\right)} & \geq\frac{\mu_{T_{j}}\left(w_{j}\right)}{N\left(w_{j}\right)+1}-\frac{1}{N\left(v_{1}\right)+2}.
\end{align*}
Then with $N\left(v_{1}\right)\geq N\left(w_{1}\right),$ it follows
that
\begin{align*}
i\left(v_{1}\right) & =\frac{R\left(v_{1}\right)}{N\left(v_{1}\right)\left(N\left(v_{1}\right)+1\right)}\\
 & >\frac{R\left(v_{1}\right)}{\left(N\left(v_{1}\right)+1\right)\left(N\left(v_{1}\right)+2\right)}\\
 & \geq\frac{\mu_{T_{j}}\left(w_{j}\right)}{N\left(w_{j}\right)+1}-\frac{1}{N\left(v_{1}\right)+2}\\
 & >\frac{\mu_{T_{j}}\left(w_{j}\right)}{N\left(w_{j}\right)+1}-\frac{1}{N\left(w_{1}\right)+1}.
\end{align*}
where the third step follows from the inequality established in the
previous paragraph. In particular, for $j=1,$
\begin{align*}
i\left(v_{1}\right) & >\frac{\mu_{T_{1}}\left(w_{1}\right)-1}{N\left(w_{1}\right)+1}\\
 & >\frac{1}{\sum_{j=1}^{d}N\left(w_{j}\right)},
\end{align*}
as $N\left(w_{j}\right)\geq4$ and $\mu_{T_{1}}\left(w_{1}\right)\geq2.$
Indeed, 
\[
N\left(w_{1}\right)\geq4\Rightarrow|T_{1}|\geq3\Rightarrow\mu_{T_{1}}\left(w_{1}\right)\geq\frac{|T_{1}|+1}{2}=2.
\]
\end{proof}
At the end of this section, we state a refining result on the leaves
of $k$-maximal subtrees.
\begin{thm}
\label{thm:branchingpath_ksubtree}Let $S\subseteq T$ be a $k$-maximal
subtree. Then the leaves of $S$ with $\deg_{T}=2$ either all lie
in core-paths, or all lie in limbs.
\end{thm}

\begin{proof}
Suppose $v$ and $w$ are leaves of $S$ both with $\deg_{T}=2,$
and $v$ lies in a core-path, $w$ lies in a limb. Let $w'$ be the
neighbour of $w$ that does not lie in $S.$ We know $i\left(S;v\right)<1/2$
and $i\left(S;w'\right)=1/2.$ Then removing $v$ and including $w'$
yields a $k$-subtree with higher local mean, which contradicts the
assumption.
\end{proof}
This theorem states that it is impossible for the degree-2-leaves
of a $k$-maximal subtree to be a mixture of the aforementioned two
kinds. To conclude, the leaves of a $k$-maximal subtree $S\subseteq T$
can be described by one of the following three situations:
\begin{enumerate}
\item[case I(a):] a mixture of leaves of $T$ and vertices in a core-path;
\item[case I(b):] a mixture of leaves of $T$ and vertices in a limb;
\item[case II:] one leaf has $\deg_{T}>2$ and all other leaves are leaves of $T.$
\end{enumerate}
Examples of case I(a) and II are given previously. For an example
of case I(b), consider the tree obtained from first connecting the
centres of two 2-stars $K_{1,2}$ by a path of two vertices and then
extending each of the four leaves by an extra pendant vertex, as shown
in Figure~\ref{fig:case2}. 

\begin{figure}[h]
\captionsetup{justification=centering}
\begin{center}     
\begin{tikzpicture}     
	{ 	
	\draw[thick] (1,0)--(2,0);
	\draw[thick] (0,0)--(1,0);
	\draw[thick] (2,0)--(3,0);
	
	\draw[thick] (3.6,0.3)--(3,0) -- (3.6,-0.3);
	\draw[thick] (4.2,0.6)--(3.6,0.3);   
	\draw[thick] (4.2,-0.6)--(3.6,-0.3);
	\draw[thick] (-0.6,-0.3)--(-1.2,-0.6);
	\draw[thick] (-0.6,0.3)--(-1.2,0.6);
	
	\draw[fill,blue] (0,0) circle (0.08);     
	\draw[fill] (1,0) circle (0.08);
	\draw[fill] (2,0) circle (0.08);
	\draw[fill] (3,0) circle (0.08);
	     
	\draw[fill] (3.6,0.3) circle (0.08);
	\draw[fill] (3.6,-0.3) circle (0.08); 
	\draw[fill] (4.2,0.6) circle (0.08);
	\draw[fill] (4.2,-0.6) circle (0.08);

	\draw[thick,blue] (-0.6,0.3)--(0,0) -- (-0.6,-0.3); 
     
	\draw[fill,blue] (-0.6,0.3) circle (0.08);
	\draw[fill] (-1.2,0.6) circle (0.08);
	\draw[fill,blue] (-0.6,-0.3) circle (0.08);
	\draw[fill] (-1.2,-0.6) circle (0.08);
	
	} 	
	\end{tikzpicture}\\

\end{center} 
\caption{An example of case I(b).}\label{fig:case2} \end{figure}
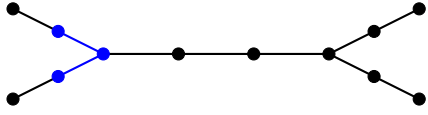

Consider the case of $k=3.$ Then the 3-subtree, shown in blue in
the figure, which contains two limb-vertices and one branching vertex
has maximal local mean at order 3 (though not unique). Indeed, other
3-subtrees that do not share the same local mean are readily seen
to have lower local mean by the Index Lemma.

\section{On indices: a detour}

The index with respect to a subtree $S$ in $T$ is well defined for
any vertex $v$ that is either a leaf of $S$ or lies outside of $S.$
In this section, we take a closer look on how indices vary among these
vertices. More specifically, for a prescribed $S,$ we compare the
index of vertex $v$ with that of its neighbour(s) that lies one step
away with respect to $S,$ i.e., an adjacent vertex whose distance
to $S$ is one greater than the distance of $v$ to $S.$ We call
such a vertex an $\mathit{outer}$ $\mathit{neighbour}$ of $v$ with
respect to $S.$ If there is no risk of ambiguity from context, then
we just say an outer neighbour of $v.$

Briefly, it is possible for an outer neighbour to have higher, lower
or equal index. A revisit to Figure~\ref{fig:exIndex} provides a
simple illustration. The two leaves to the right of $w$ both have
index $\frac{1}{2},$ while $i\left(w;S\right)=\frac{2}{5}<\frac{1}{2},$
and $i\left(v;S\right)=\frac{13}{30}>\frac{2}{5}.$ 
\begin{thm}
\label{thm:indexchange_branching}Let $S\subsetneq T$ be a proper
subtree and $v$ be either a vertex outside of $S$ or a leaf of $S.$
Let $d\geq2$ and $v_{1},\ldots,v_{d}$ be outer neighbours of $v,$
as in Figure~\ref{fig:severalfarther}. Then $i\left(v\right)<i\left(v_{j}\right)$
for any $j\in\left\{ 1,\ldots,d\right\} .$ 
\end{thm}

\begin{proof}
Let $R_{i}:=R\left(v_{i};v\right)$ and $N_{i}:=N\left(v_{i};v\right)$
for all $i\in\left\{ 1,\ldots,d\right\} .$ Let $R:=R\left(v;S\right)$
and $N:=N\left(v;S\right).$ The claim is equivalent to
\begin{equation}
\frac{1+\sum_{i=1}^{d}\frac{R_{i}}{N_{i}+1}}{\prod_{i=1}^{d}\left(N_{i}+1\right)+1}<\frac{R_{j}}{N_{j}\left(N_{j}+1\right)}.\label{eq: eqinproof3.4}
\end{equation}
As in the proof of Theorem \ref{thm:mainthm}, we exclude two cases
for $v_{j}:$ either that it is a leaf in $T,$ or that the only neighbour
of $v_{j}$ apart from $v$ is a leaf in $T.$ In both cases, the
right hand side of (\ref{eq: eqinproof3.4}) is $\frac{1}{2},$ which
is the maximal value of the index, thus the inequality automatically
holds. Hence we continue with the assumption $R_{j}/N_{j}\geq2.$
We relax the left hand side slightly,
\begin{align*}
LHS & <\frac{1+\sum\frac{R_{i}}{N_{i}+1}}{\prod\left(N_{i}+1\right)}=\frac{1+\frac{R_{j}}{N_{j}+1}+\sum_{i\text{\ensuremath{\neq}}j}\frac{R_{i}}{N_{i}+1}}{\prod\left(N_{i}+1\right)}\leq\frac{1+\frac{R_{j}}{N_{j}+1}+\sum_{i\text{\ensuremath{\neq}}j}\frac{N_{i}}{2}}{\prod\left(N_{i}+1\right)},
\end{align*}
where the summations and multiplications are taken over $i=1,\ldots,d.$
So it suffices to prove
\begin{align*}
\frac{1+\frac{R_{j}}{N_{j}+1}+\sum_{i\text{\ensuremath{\neq}}j}\frac{N_{i}}{2}}{\prod\left(N_{i}+1\right)} & \leq\frac{R_{j}}{N_{j}\left(N_{j}+1\right)}\\
1+\frac{R_{j}}{N_{j}+1}+\sum_{i\text{\ensuremath{\neq}}j}\frac{N_{i}}{2} & \leq\frac{R_{j}}{N_{j}}\prod_{i\neq j}\left(N_{i}+1\right)\\
1+\frac{1}{2}\sum_{i\text{\ensuremath{\neq}}j}N_{i} & \leq\frac{R_{j}}{N_{j}}-\frac{R_{j}}{N_{j}+1}+\frac{R_{j}}{N_{j}}\left(\sum_{i\text{\ensuremath{\neq}}j}N_{i}+\left[\textrm{higher order terms}\right]\right)\\
1 & \leq i\left(v_{j}\right)+\left(\frac{R_{j}}{N_{j}}-\frac{1}{2}\right)\sum_{i\text{\ensuremath{\neq}}j}N_{i}+\frac{R_{j}}{N_{j}}\cdot\left[\textrm{higher order terms}\right].
\end{align*}
Since $R_{j}/N_{j}\geq2$ and $d\geq2$ ensures the existence of at
least one $i\neq j,$ we have 
\[
\left(\frac{R_{j}}{N_{j}}-\frac{1}{2}\right)N_{i}\geq\frac{3}{2}\cdot1>1,
\]
and the main inequality follows.
\end{proof}
\begin{figure}[h]
\captionsetup{justification=centering}
\begin{minipage}[b]{.45\linewidth} 
\begin{center}     
\begin{tikzpicture}     
	{ 	
	\node at (-1.8,0) {$\ldots$};
	\node at (-1.8,0.25) {$S$};
	\draw[fill] (0,0) circle (0.08);     
	\draw[fill] (1,0) circle (0.08);
	\draw[thick] (-0.4,0)--(0,0) -- (1,0);
	\node at (-0.75,0) {$\ldots$};
	\node at (1.5,0) {$\ldots$};
	\node at (0,0.3) {$v$};
	\node at (1,0.3) {$w$};
	\draw[] (-1.8,0) ellipse (0.7cm and 0.5cm);	
	
	} 	
	\end{tikzpicture}\\ 
\subcaption{ Only one outer neighbour $w$: \newline $i(v)$ $\geq$ $i(w)$} 
\label{fig:onefarther} 
\end{center} 
\end{minipage}\quad\begin{minipage}[b]{.45\linewidth}
\begin{center}     
\begin{tikzpicture}     
	{ 	
	\node at (-1.8,0) {$\ldots$};
	\node at (-1.8,0.25) {$S$};
	\draw[fill] (0,0) circle (0.08);     
	\draw[fill] (1,0.4) circle (0.08);
	\draw[fill] (1,-0.4) circle (0.08);
	\draw[thick] (1,-0.4)--(0,0) -- (1,0.4);
	\node at (0,0.3) {$v$};
	\node at (1.4,0.4) {$v_1$};
	\node at (1.4,-0.4) {$v_d$};
	\draw[thick] (-0.4,0)--(0,0);
	\node at (-0.75,0) {$\ldots$};
	\node at (1,0.1) {$\vdots$};	
	\node at (1.5,0) {$\ldots$};
	\draw[] (-1.8,0) ellipse (0.7cm and 0.5cm);
	
	} 	
\end{tikzpicture}\\ 
\subcaption{Several outer neighbours $v_i$: \newline $i(v)<i(v_i)$} 
\label{fig:severalfarther} 
\end{center} 
\end{minipage} 

\caption{Change of indices with respect to $S$.}\label{fig:indexchange} \end{figure}
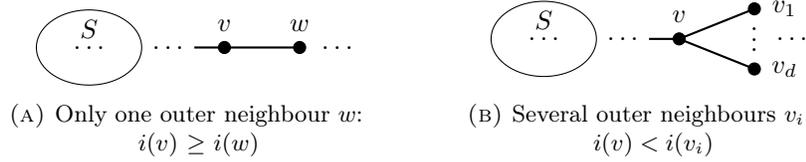
\begin{thm}
\label{thm: indexchanging_path}With notations as above, if $d=1,$
i.e., $v$ has only one outer neighbour $w$ in $T-S,$ as in Figure~\ref{fig:onefarther},
then $i\left(v\right)\geq i\left(w\right).$ The equality holds if
and only if $i\left(v\right)=i\left(w\right)=\frac{1}{2}.$
\end{thm}

\begin{proof}
Let $N:=N\left(w;S\right)$ and $R:=R\left(w;S\right).$ Reversing
the inequality in (\ref{eq: eqinproof3.4}) and setting $d=1$ yields
the following equivalent inequalities:
\begin{align*}
\frac{1+\frac{R}{N+1}}{\left(N+1\right)+1} & \geq\frac{R}{N\left(N+1\right)}\\
\frac{R+N+1}{N+2} & \geq\frac{R}{N}\\
RN+N^{2}+N-RN-2R & \geq0\\
N^{2}+N-2R & \geq0\\
\frac{1}{2} & \geq\frac{R}{N\left(N+1\right)}.
\end{align*}
Then the claim follows as the right hand side is just the index of
$w$ with respect to $S.$
\end{proof}
The picture of how indices change can be described as follows: assume
that we are traveling away from $S$ (towards leaves of $T$) and
standing at $v.$ If $\deg_{T}(v)=2,$ then the (only) outer neighbour
has a lower or equal index; if $\deg_{T}(v)>2,$ then all outer neighbours
have higher indices. The situation is summarized in Figure~\ref{fig:indexchange}.

In the last part of this section, we give an easier proof to the fact
that the vertex with the greatest local mean has degree 1 or 2, which
was first proved by Wagner and Wang \cite[Theorem 3.2]{MR3433637}.
\begin{thm}
\label{thm:oldthm}The vertex that attains maximal local mean has
degree 1 or 2.
\end{thm}

\begin{proof}
Let $v\in T$ be the vertex that attains the maximal local mean. Assume
for the purpose of contradiction that $v$ has at least three neighbours
and let $a,b$ and $c$ be three of them.

Without loss of generality, assume that $a$ has the highest index
with respect to $v,$ i.e.
\[
i\left(a;v\right)=\max\left\{ i\left(a;v\right),i\left(b;v\right),i\left(c;v\right)\right\} .
\]
As $i\left(b;v\right)$ is defined with respect to the subtree that
is rooted at $b$ and grows away from $v,$ it is equivalent to write
it as $i\left(b;a\right).$ Hence the above implies
\[
i\left(a;v\right)\geq i\left(b;a\right)\textrm{ and }i\left(a;v\right)\geq i\left(c;a\right).
\]
The strategy is to show that including $a$ and then deleting $v$
yields a higher local mean. By the Index Lemma, it is equivalent to
show that
\[
i\left(a;v\right)>i\left(v;a\right).
\]
But Theorem \ref{thm:branchingpath} gives
\[
i\left(b;a\right)>i\left(v;a\right)\textrm{ and }i\left(c;a\right)>i\left(v;a\right),
\]
which completes the proof.
\end{proof}

\section{On extremal local means: a revisit}

Going back to the local means, the analysis on indices provides a
simple proof of Theorem \ref{thm:mainthm}. Suppose $S$ has two leaves
$v$ and $w$ with $\deg_{T}>2.$ Without loss of generality, let
$i\left(v\right)\geq i\left(w\right).$ Then by Theorem \ref{thm:indexchange_branching},
any outer neighbour of $v$ has a higher index, thus also higher than
that of $w,$ which further implies that replacing $w$ by any outer
neighbour of $v$ yields a subtree of the same order with higher local
mean. This process only terminates when $S$ has no more than one
leaf with $\deg_{T}>2.$

The parallel statement of Theorem \ref{thm:mainthm} for the minimizing
case can also be proved by the above method. In general, a subtree
that minimizes the local mean tends to have its leaves located at
vertices with $\deg_{T}>2,$ contrary to the maximizing case. However,
it may depend on the order of subtrees $k,$ order of $T$ and the
structure of $T.$ When $k$ is big enough, all $k$-subtrees, including
the minimizing one(s), may have to contain more than one vertex with
$\deg_{T}\leq2.$ For example, in the case of a star with $n$ leaves
and $k>2,$ any $k$-subtree has $k-1$ leaves.
\begin{thm}
\label{thm:minimal_case}Let $T^{*}\neq\emptyset$ be the core of
$T.$ Let $S\subseteq T$ with $|S|=k,$ and $k\leq|T^{*}|.$ If $S$
is $k$-minimal, i.e. $\mu_{T}\left(S\right)$ attains the minimum
of the local mean at order $k,$ then $S$ has no more than one leaf
with $\deg_{T}\leq2$ and at least one leaf with $\deg_{T}\ge3.$
Specifically, for $k=1,$ the vertex that attains minimal local mean
has $\deg_{T}\geq3.$
\end{thm}

\begin{proof}
Let $S$ be $k$-minimal in $T.$ First, we show that $S\subseteq T^{*}.$
Assume that $S$ contains a leaf $v$ that lies in some limb. Then
$i(v;S)=\frac{1}{2}.$ Note that $k\leq|T^{*}|$ implies that the
rest of $S$ cannot contain the entire core, which further implies,
by the non-emptiness of $T^{*},$ that there exists a neighbour $w$
of $S,$ such that the component of $w$ in $T-S$ contains at least
one vertex with degree $>2.$ Therefore, this component is not a path,
which means $i(w;S)<\frac{1}{2}.$ Thus one can include the neighbour
$w$ and delete $v,$ resulting in a subtree of the same order but
lower local mean, which leads to a contradiction. So, $S\subseteq T^{*}.$ 

Now consider $k\geq2.$ Assume that $S$ has two leaves $v$ and $w$
with $\deg_{T}=2,$ and without loss of generality, $i\left(v\right)\leq i\left(w\right)<\frac{1}{2}.$
By Theorem \ref{thm: indexchanging_path}, we know that the only outer
neighbour of $v,$ call it $v',$ has a lower index than $v,$ hence
lower than the index of $w.$ Including $v'$ and deleting $w$ results
in a subtree of order $k$ but with local mean lower than $S,$ and
the desired contradiction follows. As $k\geq2,$ the previous argument
immediately implies that $S$ has at least one leaf with $\deg_{T}>2$
since $S$ has at least two leaves. 

For $k=1,$ the process is similar. Let $S=\{a\}$ and denote the
local mean at $S$ by $\mu.$ Assume $\deg_{T}\left(a\right)=2.$
Let $v,w$ be its two \foreignlanguage{british}{neighbours}. With
out loss of generality, let $i\left(v;a\right)\ge i\left(w;a\right).$
Observe that both indices are strictly less than $\frac{1}{2}$ since
$a\in T^{*}.$ First, we include $w$ and obtain a 2-subtree $\{a,w\}$
with local mean $\mu+i(w;a).$ Then we delete $a$ and obtain a singleton
subtree $\{w\}$ with local mean $\mu+i(w;a)-i(a;w).$ However, $i(a;w)>i(v;w)=i(v;a)$
by Theorem \ref{thm: indexchanging_path}. Hence $w$ has a lower
local mean than $a,$ which contradicts that $S=\{a\}$ is $k$-minimal.
\end{proof}
Note that since there is no core in a path, the above result does
not apply to paths.

\begin{example*}
Let us consider an example of two $\left(n+1\right)$-stars $K_{1,n}$
with their \foreignlanguage{british}{centres} connected by a path
of $k-2$ vertices, a more general setting than the example in Section
3 (Figure~\ref{fig:bothcase}). Let $n>1$ and $k>2.$ We will show
that a subtree minimizing the local mean at order 2 has one vertex
at one of the \foreignlanguage{british}{centres} of the stars and
the other vertex in the middle path, which agrees with the prediction
of the previous results.
\end{example*}
Let vertices be labelled as in Figure~\ref{fig:twostars}.

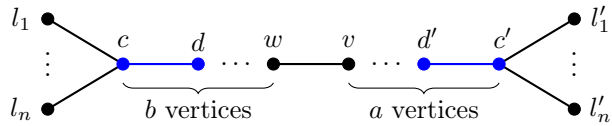
\begin{figure}[h]
\captionsetup{justification=centering}    
\begin{center}     
\begin{tikzpicture}     
	{          
	\draw[thick] (0,0)--(1,0);
	\node at (1.5,0) {$\ldots$};
	\draw[thick] (4,0.6)--(3,0) -- (4,-0.6); 
	\draw[fill,blue] (2,0) circle (0.08);
	\draw[thick,blue] (2,0)--(3,0); 
	
	\draw[fill] (0,0) circle (0.08);     
	\draw[fill] (1,0) circle (0.08);
	\draw[fill,blue] (3,0) circle (0.08);     
	\draw[fill] (4,0.6) circle (0.08);
	\node at (4,0.1) {$\vdots$};
	\draw[fill] (4,-0.6) circle (0.08); 

	\node at (-0.5,0) {$\ldots$};
	\draw[thick] (-3,0.6)--(-2,0) -- (-3,-0.6); 
	\draw[fill,blue] (-1,0) circle (0.08);

	\draw[fill,blue] (-2,0) circle (0.08);     
	\draw[fill] (-3,0.6) circle (0.08);
	\node at (-3,0.1) {$\vdots$};
	\draw[fill] (-3,-0.6) circle (0.08);
	\draw[thick,blue] (-1,0)--(-2,0);
	
    \node at (-3.35,0.6) {$l_1$};
	\node at (-3.35,-0.6) {$l_n$};
	\node at (-2,0.3) {$c$};
	\node at (-1,0.3) {$d$};
	\draw [decorate,decoration={brace,amplitude=4pt},xshift=0pt,yshift=0pt] (3,-0.25) -- (1,-0.25) node [black,midway,yshift=-0.35cm]{$a$ vertices};
	
	\node at (4.35,0.6) {$l'_1$};
	\node at (4.35,-0.6) {$l'_n$};
	\node at (3.05,0.35) {$c'$};
	\node at (2.05,0.35) {$d'$};
	\draw [decorate,decoration={brace,amplitude=4pt},xshift=0pt,yshift=0pt] (0,-0.25) -- (-2,-0.25) node [black,midway,yshift=-0.35cm]{$b$ vertices};

	\node at (1,0.3) {$v$};     
	\node at (0,0.3) {$w$};
	   } 	
\end{tikzpicture}\\ 
\end{center} 
\caption{Two 2-minimal subtrees (marked in blue).}
\label{fig:twostars} 
\end{figure} 

There are two types of 2-subtrees to calculate. The first type consists
of one centre and one leaf, without loss of generality, say $S_{1}=\left\{ l_{1},c\right\} .$
The second type is an edge in the path, $S_{2}=\left\{ w,v\right\} .$
But according to the Index Lemma $\left\{ l_{1},c\right\} $ has a
higher local mean than $\left\{ c,d\right\} ,$ since $i\left(l_{1};c\right)=\frac{1}{2}>i\left(d;c\right).$
So we only need to calculate the second type. Denote the local mean
at $S_{2}$ by $\mu_{2}.$ Then $\mu_{2}=\mu_{v}+\mu_{w},$ where
$\mu_{v}$ is the local mean at $v$ of the component in $T-wv$ that
contains $v,$ and similarly for $\mu_{w}.$ We have
\[
\mu_{w}=\frac{2^{n}+n2^{n-1}+2^{n}\left(b-1\right)+\frac{b\left(b-1\right)}{2}}{2^{n}+b-1}=\frac{b}{2}+\frac{\left(n+b\right)2^{n-1}}{2^{n}+b-1},
\]
and analogously
\[
\mu_{v}=\frac{a}{2}+\frac{\left(n+a\right)2^{n-1}}{2^{n}+a-1}.
\]
Thus
\[
\mu_{2}=\frac{k}{2}+\frac{\left(n+a\right)2^{n-1}}{2^{n}+a-1}+\frac{\left(n+b\right)2^{n-1}}{2^{n}+b-1},
\]
where $1\leq a,b\leq k-1$ and $a+b=k.$ We claim that $\mu_{2}$
attains its minimum

\[
\left(\mu_{2}\right)_{min}=\frac{k}{2}+\frac{n+1}{2}+\frac{\left(n+k-1\right)2^{n-1}}{2^{n}+k-2},
\]
at either end $a=1,b=k-1$ or $a=k-1,b=1.$ Then the minimum of the
local mean at order 2 is attained by subtree $\{c,d\}$ or $\{c',d'\},$
which agrees with the theory.
\begin{proof}[Proof of the claim]
 To see that $\mu_{2}$ attains its minimum at either end, i.e. $a=1,b=k-1$
or $a=k-1,b=1,$ we consider the following function,
\begin{align*}
f\left(x\right) & =\frac{\left(n+x\right)2^{n-1}}{2^{n}+x-1}+\frac{\left(n+k-x\right)2^{n-1}}{2^{n}+k-x-1}\\
 & =2^{n}-\frac{2^{2n-1}-\left(n+1\right)2^{n-1}}{x+2^{n}-1}+\frac{2^{2n-1}-\left(n+1\right)2^{n-1}}{x-2^{n}-k+1},
\end{align*}
where $1\leq x\leq k-1.$ Taking the derivative, one has
\[
f'\left(x\right)=A_{n}\left(\frac{1}{\left(2^{n}+x-1\right)^{2}}-\frac{1}{\left(2^{n}+k-1-x\right)^{2}}\right),
\]
where $A_{n}=2^{2n-1}-\left(n+1\right)2^{n-1}>0$ for all $n>1.$
It immediately follows that
\[
\begin{cases}
f'\left(x\right)>0, & 1\leq x<\frac{k}{2},\\
f'\left(x\right)=0, & x=\frac{k}{2},\\
f'\left(x\right)<0, & \frac{k}{2}<x\leq k-1.
\end{cases}
\]
So for $1\leq x\leq k-1,$ $f(x)$ attains its minimum at $x=1$ and
$x=k-1,$ which completes the proof. Also observe that $f\left(x\right)$
attains its maximum in the middle of the range, which justifies the
claim in the previous example (Figure~\ref{fig:bothcase}).
\end{proof}

\section{On density and its localization}

For the sake of comparing local means of subtrees of different orders,
we normalize the local mean and define the local density of $T$ which
generalizes the (global) density in \cite{MR735190}. 
\begin{defn*}
Let $S\subset T$ be a proper subtree where $|T|=n$ and $|S|=k.$
Let $\mu_{T}\left(S\right)$ be the local mean at $S.$ Define the
$\mathit{local}$ $\mathit{density}$ $\mathit{at}$ $\mathit{S}$
as
\[
D_{T}\left(S\right)=\frac{\mu_{T}\left(S\right)-k}{n-k}.
\]
\end{defn*}
Just as density can be interpreted as the probability of a randomly
picked vertex being in a randomly picked subtree, local density at
$S$ can be interpreted as the probability of a random vertex $\mathit{outside}$
of $S$ being in a random subtree $\mathit{containing}$ $S.$
\begin{rem*}
In the case of $S=\emptyset,$ namely $k=0,$ the local density reduces
to the global density. At the other end, when $k=n-1,$ there are
only two subtrees containing $S,$ therefore $\mu\left(S\right)=k+\frac{1}{2}$
and $D\left(S\right)=\frac{1}{2}.$ 
\end{rem*}
\begin{thm}
\label{thm:DSvsDS-w}With the notations above, let $2\leq k\leq n-1.$
Let $w$ be a leaf of $S$ and $v$ be a neighbour of $S.$ Then
\[
D\left(S\right)\geq D\left(S-w\right)\Leftrightarrow D\left(S\right)\geq1-i\left(w;S\right),
\]
where one equality holds if and only if the other equality holds.
Similarly,
\[
D\left(S+v\right)\geq D\left(S\right)\Leftrightarrow D\left(S\right)\geq1-i\left(v;S\right).
\]
\end{thm}

\begin{proof}
By the Index Lemma, we have the following equivalent inequalities:
\begin{align*}
D\left(S\right) & \geq D\left(S-w\right)\\
\frac{\mu\left(S\right)-k}{n-k} & \geq\frac{\mu\left(S\right)-i\left(w\right)-\left(k-1\right)}{n-\left(k-1\right)}\\
\frac{n-k+1}{n-k} & \geq\frac{\mu\left(S\right)-k-i\left(w\right)+1}{\mu\left(S\right)-k}\\
\frac{1}{n-k} & \geq\frac{1-i\left(w\right)}{\mu\left(S\right)-k}\\
D\left(S\right) & \geq1-i\left(w\right).
\end{align*}
The same calculation yields the other equivalence.
\end{proof}
We know that the local mean trivially increases if the subtree absorbs
a neighbour, but for local density, this is no longer the case. Local
density can increase after deleting a leaf or decrease after including
a neighbour. An advantage, however, is that local density allows one
to compare subtrees of $\mathit{different}$ orders. Hence one can
say ``the subtree that attains maximal local density in $T$'' without
specifying the order of subtree unlike the case of local mean.
\begin{example*}
Consider $T$ to be the tree obtained by connecting the centres of
two 2-stars $K_{1,2}$ by an edge. See Figure~\ref{fig:two2stars}.

\begin{figure}[h]     
\begin{center}     
\begin{tikzpicture}     
	{          
	\draw[thick] (0,0)--(1,0);
	
	\draw[thick] (2,0.4)--(1,0) -- (2,-0.4);    
	
	\draw[fill] (0,0) circle (0.08);     
	\draw[fill] (1,0) circle (0.08);
	     
	\draw[fill] (2,0.4) circle (0.08);
	\draw[fill] (2,-0.4) circle (0.08); 

	\draw[thick] (-1,0.4)--(0,0) -- (-1,-0.4); 
     
	\draw[fill] (-1,0.4) circle (0.08);

	\draw[fill] (-1,-0.4) circle (0.08);
	
    \node at (-1.35,0.4) {$v$};
	\node at (0,0.3) {$c$};
	\node at (1,0.3) {$d$};

	   } 	
\end{tikzpicture}\\ 
\end{center} 
\caption{Two 2-stars.}
\label{fig:two2stars} 
\end{figure}
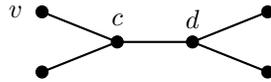 

Then $T^{*}=\left\{ c,d\right\} $ and we have
\[
D\left(v,c\right)=\frac{21}{40},D\left(c\right)=\frac{13}{25},D\left(v\right)=\frac{31}{55}.
\]
Thus
\[
D\left(c\right)<D\left(v,c\right)<D\left(v\right).
\]
As this example shows, local density can change in both directions
by removing a vertex, in other words, local density does not respect
set-inclusion in general. Furthermore, there are in fact four subtrees
that attain maximal local density, namely the four singleton leaves.
The local densities of the remaining subtrees (those not isomorphic
to $\{v,c\},\{c\}$ or $\{v\}$) are indeed all equal to $\frac{1}{2},$
which is justified by the next proposition.
\end{example*}
\begin{prop}
\label{prop:bodydensity1/2}Let $T^{*}$ be the core of $T$ and $S\subset T$
a non-empty subtree with $T^{*}\subseteq S.$ Then
\[
D\left(S\right)=\frac{1}{2}.
\]
\end{prop}

\begin{proof}
From the last remark we know that for any $\left(n-1\right)$-subtree
the local density is $\frac{1}{2}.$ If there is no limb vertex in
it, the statement is true automatically. Otherwise, one can take away
one limb vertex that is also a leaf of $S$, say $v,$ and
\[
D\left(S\right)=\frac{1}{2}=1-i\left(v;S\right),
\]
where $i\left(v;S\right)=\frac{1}{2}$ since $S$ contains the body
of $T$ and $v$ is a limb vertex. Then Theorem \ref{thm:DSvsDS-w}
implies $D\left(S-v\right)=D\left(S\right).$ Repeating the above
process until all limbs are removed yields the statement.
\end{proof}
Combining Proposition \ref{prop:bodydensity1/2} and Theorem \ref{thm:DSvsDS-w},
one immediately has the following corollary.
\begin{cor}
\label{cor:bodyminusonenode}Let $T^{*}$ be the core of $T.$ Let
$v$ be a leaf of $T^{*}\neq\left\{ v\right\} .$ Then
\[
D\left(T^{*}-v\right)>\frac{1}{2}.
\]
\end{cor}

Next, we establish a lower bound on the local density.
\begin{thm}
\label{thm:localdensitylowerbound}Let $S\neq\emptyset$ be a non-empty
subtree of $T.$ Then 
\[
D\left(S\right)\geq\frac{1}{2},
\]
with equality if and only if $S\supseteq T^{*}.$
\end{thm}

\begin{proof}
Let $k=\left|S\right|.$ By contraction and its induced bijection
and the equality (\ref{eq:5}), we have
\[
D_{T}\left(S\right)=\frac{\mu_{T}\left(S\right)-k}{n-k}=\frac{\mu_{T/S}\left(S/S\right)-1}{\left(n-k+1\right)-1}=D_{T/S}\left(S/S\right),
\]
i.e., the local density is invariant under contraction. Then (\ref{eq:lowerbound})
implies
\begin{align*}
D_{T/S}\left(S/S\right) & =\frac{\mu_{T/S}\left(S/S\right)-1}{n-\left(k-1\right)-1}\geq\frac{\frac{n-\left(k-1\right)+1}{2}-1}{n-k}=\frac{1}{2},
\end{align*}
with equality if and only if $T/S$ is astral over $S/S,$ which is
equivalent to $S\supseteq T^{*}.$
\end{proof}
Despite the fact that it is possible for the local density to increase
or decrease in removing a vertex, with the above theorem, one can
now say that the local density (at $S$) does not increase in removing
a vertex of index $\frac{1}{2}$ (with respect to $S$).
\begin{prop}
\label{thm:remove1/2node}Let $S$ consist of two adjacent vertices
$v,w$ with $i\left(w;v\right)=\frac{1}{2}.$ Then
\[
D_{T}\left(v\right)\leq D_{T}\left(v,w\right),
\]
where the equality holds if and only if $T$ is astral over $v.$
\end{prop}

\begin{proof}
This follows directly from the first half of Theorem \ref{thm:DSvsDS-w}
and Theorem \ref{thm:localdensitylowerbound}.
\end{proof}
This proposition can easily be extended to a more general setting
as follows. 
\begin{thm}
\label{thm:limbsincreaselocald}Let $S\varsubsetneq T$ and $S^{*}=S\cap T^{*}.$
If $S^{*}\neq\emptyset,$ then
\[
\ensuremath{D\left(S^{*}\right)\leq D\left(S\right),}
\]
where the equality holds if and only if $S\subseteq T^{*}$ or $T^{*}\subseteq S.$
\end{thm}

\begin{proof}
Let $L=S-S^{*}$ be the part of $S$ consisting of limbs of $T.$
If $L=\emptyset,$ then $S=S^{*},$ i.e., $S\subseteq T^{*},$ and
the conclusion follows trivially. Now let $v\in L$ be a leaf of $S$
in some limb. By Theorem \ref{thm:localdensitylowerbound}, $D(S)\ge\frac{1}{2}$
with equality if and only if $T^{*}\subseteq S.$ On the other hand,
$i(v;S)=\frac{1}{2}.$ Therefore, $D(S)\ge1-i(v;S).$ Then Theorem
\ref{thm:DSvsDS-w} yields $D(S-v)\le D(S),$ with equality if and
only if $T^{*}\subseteq S.$ One can repeat this process of deleting
limb leaves of $S$ and eventually reach the conclusion.
\end{proof}
From the results so far, we know that the lower bound $D\left(S\right)=\frac{1}{2}$
is attained if and only if $S$ contains the core of $T.$ Any other
subtree has $D\left(S\right)>\frac{1}{2}.$ 

On the other hand, local density can be arbitrarily close to 1. Let
$T$ be a tree of order $n.$ Indeed, we know that local mean at a
vertex is no less than global mean \cite[Theorem 3.9]{MR735190},
and that there exist trees with global mean arbitrarily close to $n.$
Combining the two facts, there exist trees whose local densities are
arbitrarily close to 1,
\[
1\geq D_{T}\left(v\right)=\frac{\mu_{T}\left(v\right)-1}{n-1}\geq\frac{\mu_{T}-1}{n-1}\rightarrow1,\;n\rightarrow\infty.
\]

However, the story of local density on the maximal side is still far
from complete. We have not yet found a characterization for subtrees
that attain maximal local density as we have for the minimal case.
What is known is that if $S$ attains the maximal local density in
$T$ then $S$ necessarily attains the maximal local mean at order
$\left|S\right|.$ This observation follows immediately from the definition
of local density. Hence the properties described in Section 3, those
of $k$-maximal subtrees, also apply to subtrees of maximal local
density. However, it is reasonable to expect that there exist properties
that $only$ subtrees of maximal local density have and $k$-maximal
subtrees generally do not.

Theorem \ref{thm:limbsincreaselocald} reveals some information about
the structure of the subtree attaining maximal local density. Let
us call a core vertex $v$ that is adjacent to some limb vertex a
$\mathit{joint\,vertex}$. Now if a subtree (not big enough to contain
the core, as this makes its local density minimal) contains some joint
vertex $v,$ then including the whole limb connected to $v$ would
only increase the local density. So for a subtree of maximal local
density and any joint vertex $v,$ it would either avoid $v$ or contain
$v$ along with its entire limb. Therefore, there are three possibilities
for the structure of a subtree with maximal local density:
\begin{enumerate}
\item A proper subtree of the core that does not contain any joint vertex.
\item A subtree that can be decomposed into two parts: a proper subtree
of the core containing some joint vertices, and their corresponding
whole limbs.
\item A proper subtree (path) of some limb that contains the leaf.
\end{enumerate}
\begin{example*}
Let us take a look at some concrete examples: the subtrees of maximal
local density in paths and stars.

The case of paths is straightforward since the path has no core. For
a path of order $n,$ by Proposition \ref{prop:bodydensity1/2}, any
non-empty proper subtree $S$ has $D\left(S\right)=\frac{1}{2}.$

Now consider the star $T=K_{1,n-1}$ with centre $c$ and a leaf $v.$
For any subtree $S$ with order at least 2, it must contain $\{c\}=T^{*}.$
Hence Proposition \ref{prop:bodydensity1/2} yields $D(c)=\frac{1}{2}$
and $D(S)=\frac{1}{2}$ for $|S|\geq2.$ Next, we calculate the densities
of singleton subtrees. Denote by $N_{c}$ the number of subtrees of
$T-\{v\}$ containing $c,$ and $R_{c}$ the total cardinality of
such subtrees. Then
\[
N_{c}=2^{n-2},R_{c}=2^{n-2}+\left(n-2\right)2^{n-3}=n2^{n-3},
\]
and equality (\ref{eq:3}) yields
\[
D\left(v\right)=\frac{\mu\left(v\right)-1}{n-1}=\frac{1+\frac{R_{c}}{N_{c}+1}-1}{n-1}=\frac{R_{c}}{\left(n-1\right)\left(N_{c}+1\right)}.
\]
Substituting $N_{c}$ and $R_{c}$ with the expressions above, we
have
\begin{equation}
D\left(v\right)=\frac{n2^{n-3}}{\left(n-1\right)\left(2^{n-2}+1\right)}\geq\frac{1}{2}=D\left(c\right),\label{eq:Dv>DcSTAR}
\end{equation}
for $n\geq4$ and equality holds in the middle inequality only when
$n=2$ or $3.$ Indeed, this inequality is equivalent to $2^{n-2}\geq n-1,$
which is true for all $n\geq2$ and equality holds if and only if
$n=2$ or $3.$

Therefore, the subtrees that attain the maximal local density in the
star $K_{1,n-1},$ for $n\geq2,$ are all the singleton subtrees consisting
of one leaf.\hfill \qedsymbol
\end{example*}
\medskip{}

The last part of this section answers the following question:
\begin{quote}
Does local density always exceed global density?
\end{quote}
From \cite{MR735190}, we know that local mean (at a vertex) exceeds
global mean, for any tree with order greater than one. The next example
eliminates the possibility for the parallel to hold in density.
\begin{example*}
Consider the star $T=K_{1,n-1}$ with $n\geq2.$ Let $v$ be a leaf
and $c$ the centre. We will compare the three quantities, the global
density $D_{T},$ the local density at a leaf $D_{T}(v)$ and the
local density at the centre $D_{T}(c).$ For the global density $D_{T},$
the number of all subtrees is $2^{n-1}+n-1,$ and the total cardinality
is 
\begin{align*}
 & \sum_{i=0}^{n-1}{n-1 \choose i}\left(i+1\right)+n-1\\
 & =2^{n-1}+\left(n-1\right)2^{n-2}+n-1\\
 & =\left(n+1\right)2^{n-2}+n-1.
\end{align*}
So
\[
D_{T}=\frac{\mu_{T}}{n}=\frac{\left(n+1\right)2^{n-2}+n-1}{n2^{n-1}+n^{2}-n}.
\]
Observe that $D_{T}>\frac{1}{2},$ since the inequality is equivalent
to $2^{n-1}>(n-1)(n-2),$ which holds for all $n\geq1.$ Then, on
the local densities, we use the result (\ref{eq:Dv>DcSTAR}) from
the previous example: 
\[
D_{T}\left(v\right)=\frac{n2^{n-3}}{\left(n-1\right)\left(2^{n-2}+1\right)}\geq\frac{1}{2}=D_{T}\left(c\right),
\]
where inequality holds strictly for $n\geq4,$ and $D_{T}(v)=\frac{1}{2}=D_{T}(c)$
for $n=2$ or $3.$

What is left is to compare $D_{T}\left(v\right)$ and $D_{T}.$ Simplification
yields that
\[
D_{T}\left(v\right)>D_{T}\;\textrm{if and only if }\;4^{n-2}+n(n-1)(n-4)2^{n-3}>(n-1)^{2},
\]
where the latter is true for all $n\geq4.$ To conclude, the three
quantities satisfy
\[
D_{T}\left(c\right)<D_{T}<D_{T}\left(v\right)
\]
for $n\geq4,$ and
\[
D_{T}\left(c\right)<D_{T}\left(v\right)<D_{T}
\]
for $n=2$ or $3.$ Hence the example shows that there is no general
order relation between local densities and the global density.\hfill \qedsymbol
\end{example*}

\section{On the local density of two-vertex subtrees}

In the last section, we study the simplest case when $S=\left\{ v,w\right\} ,$
where $v,w$ are adjacent, and compare $D_{T}\left(v,w\right),D_{T}\left(v\right)$
and $D_{T}\left(w\right).$ 
\begin{defn*}
Let $T_{v}$ be a tree with root at $v$ and order greater than 1.
Call $T_{v}$ of $\mathit{high\,index\,type},$ or $\mathit{type}$
$H,$ if it satisfies 
\[
D_{T_{v}}\left(v\right)\geq1-i\left(T_{v}\right).
\]
Call a rooted tree of $\mathit{low\,index\,type},$ or $\mathit{type}$
$L,$ if
\[
D_{T_{v}}\left(v\right)<1-i\left(T_{v}\right).
\]
\end{defn*}
\begin{rem*}
The restriction on the order is due to the fact that local density
is not defined for a singleton tree. In fact, for any tree $T,$ $D_{T}\left(T\right)$
is not defined.
\end{rem*}
Direct calculation from the definition implies that $T_{v}$ being
of type L is equivalent to 
\begin{equation}
\mu_{T_{v}}\left(v\right)<\frac{N_{v}n_{v}+n_{v}}{N_{v}+n_{v}},\label{eq:typeL}
\end{equation}
where $N_{v}$ is the number of subtrees rooted at $v$ and $n_{v}=\left|T_{v}\right|.$ 
\begin{thm}
\label{thm:deg2typeL}Let $T_{v}$ be a tree rooted at $v$ with $\deg(v)>1.$
Then $T_{v}$ is of type L.
\end{thm}

\begin{proof}
By assumption, $T_{v}$ can be split into two rooted subtrees $A_{v}$
and $B_{v},$ both with root $v$ and order greater than 1. See Figure~\ref{fig:typeL}. 

\begin{figure}[h]     
\begin{center}     
\begin{tikzpicture}     
	{          
	    
	\draw[fill] (0,0) circle (0.05);  

	\node[thick,rotate=25,anchor=center] at (0.55,0.3) {$\ldots$}; 
	\node[thick,rotate=-25,anchor=center] at (0.55,-0.3) {$\ldots$};

	\draw[dashed,rotate=25,anchor=center] (1.3,0.3) arc (90:270:1.6cm and 0.3cm);
	\draw[dashed,rotate=-25,anchor=center] (1.3,0.3) arc (90:270:1.6cm and 0.3cm);
	
    \node at (0,0.5) {$v$};
	\node at (-0.7,0) {$T_v$};
	\node[rotate=25,anchor=center] at (1.1,0.55) {$A$};
	\node[rotate=-25,anchor=center] at (1.1,-0.5) {$B$};   
	     } 	
\end{tikzpicture}\\ 
\end{center} 
\caption{$T_v$ splits into rooted subtrees $A$ and $B$.}
\label{fig:typeL} 
\end{figure}
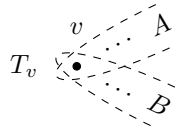 

If $a=|A_{v}|$ and $b=|B_{v}|,$ then $a,b\geq2$ and $a+b=n+1,$
where $n=|T_{v}|\geq3.$ Observe that for any rooted tree of order
$m$ and any integer $k\in\left\{ 1,\ldots,m\right\} ,$ there exists
at least one rooted subtree of order $k.$ Also note that each pair
of rooted subtrees of $A_{v}$ and $B_{v},$ of orders $i$ and $j$
respectively, corresponds to a subtree of $T_{v}$ rooted at $v$
of order $i+j-1.$ Combining these two observations, denoting $R=R_{T_{v}}\left(v\right)$
and $N=N_{T_{v}}\left(v\right),$ we have
\begin{align*}
R & \leq\sum_{i=1}^{a}\sum_{j=1}^{b}\left(i+j-1\right)+\left(N-ab\right)\left(n-1\right)\\
 & =\frac{ab}{2}\left(a+b\right)+\left(N-ab\right)\left(n-1\right)\\
 & =N\left(n-1\right)+ab\left(\frac{a+b}{2}-n+1\right)\\
 & =N\left(n-1\right)+ab\left(\frac{n+1}{2}-n+1\right)\\
 & =N\left(n-1\right)-ab\frac{n-3}{2}\\
 & \leq N\left(n-1\right)-\left(n-1\right)\left(n-3\right),
\end{align*}
where the last inequality follows from $n\geq3$ and the fact that
$a+b=n+1$ and $a,b\geq2,$ thus $ab\geq2\left(n-1\right).$ Write
$\mu=\mu_{T_{v}}\left(v\right),$ then
\begin{align*}
\frac{Nn+n}{N+n}-\mu & =\frac{Nn+n}{N+n}-\frac{R}{N}\\
 & \geq\frac{Nn+n}{N+n}-\frac{N\left(n-1\right)-\left(n-1\right)\left(n-3\right)}{N}\\
 & =\frac{N\left(N-2n+3\right)+n\left(n-1\right)\left(n-3\right)}{N\left(N+n\right)}\\
 & \overset{(*)}{\geq}\frac{N}{N\left(N+n\right)}\\
 & =\frac{1}{N+n}>0,
\end{align*}
where step $(*)$ follows from $N\geq ab\geq2\left(n-1\right).$
\end{proof}
\begin{thm}
\label{thm:typeH}Let $T_{v}$ be a type H tree rooted at $v$ with
order at least 2. Let $T_{v'}$ be the tree rooted at $v'$ obtained
by connecting a pendant vertex $v'$ to $v$ of $T_{v},$ as in Figure~\ref{fig:typeH}.
Then $T_{v'}$ is of type H.
\end{thm}

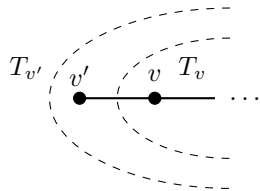
\begin{figure}[h]     
\begin{center}     
\begin{tikzpicture}     
	{          
	\draw[thick] (0,0)--(1,0) -- (1.8,0);    
	\draw[fill] (0,0) circle (0.08);     
	\draw[fill] (1,0) circle (0.08);
	\draw[dashed] (2,1.2) arc (90:270:2.4cm and 1.2cm);
	\draw[dashed] (2,0.8) arc (90:270:1.5cm and 0.8cm);
	
	\node at (-0.7,0.4) {$T_{v'}$};
	\node at (1.5,0.4) {$T_v$};
    \node at (0,0.3) {$v'$};
	\node at (1,0.3) {$v$};
	\node at (2.2,0) {$\ldots$};    
	     } 	
\end{tikzpicture}\\ 
\end{center} 
\caption{$T_v$ and $T_{v'}$.}
\label{fig:typeH} 
\end{figure} 
\begin{proof}
Let $N$ be the number of subtrees of $T_{v}$ rooted at $v$ and
$R$ the total cardinality of these subtrees. Denote the corresponding
quantities for $T_{v'}$ by $N'$ and $R'$. We know $N'=N+1,$ $R'=R+N+1$
and $n'=n+1.$ The conclusion is equivalent to 
\begin{align}
\mu_{T_{v'}}\left(v'\right)=\frac{R'}{N'} & \geq\frac{N'n'+n'}{N'+n'}\nonumber \\
N'^{2}n'+N'n'-R'\left(N'+n'\right) & \leq0,\label{eq:ineqinproof5.3}
\end{align}
and the assumption is equivalent to
\begin{equation}
N^{2}n+Nn-R\left(N+n\right)\leq0.\label{eq:ineqinproof5.3_2}
\end{equation}
Expanding the left hand side of (\ref{eq:ineqinproof5.3}) in terms
of $N,R$ and $n,$ we have 
\begin{align*}
 & \left(N+1\right)^{2}\left(n+1\right)+\left(N+1\right)\left(n+1\right)-\left(R+N+1\right)\left(N+n+2\right)\\
 & =N^{2}n+2Nn-R\left(N+n\right)+n-2R\\
\textrm{(by \eqref{eq:ineqinproof5.3_2})} & \leq Nn+n-2R\\
 & \leq0,
\end{align*}
where the last step follows from an application of inequality (\ref{eq:lowerbound}).
\end{proof}
Now we focus on the main task of this section. Let $S=\left\{ v,w\right\} $
be a proper subtree of $T$ where $v,w$ are adjacent. Our goal is
to compare $D_{T}\left(v,w\right),$ $D_{T}\left(v\right)$ and $D_{T}\left(w\right).$
Note that at most one of $v,w$ can be a leaf of $T$ since $T\neq S.$
First we consider the case that $v$ is a leaf of $T,$ as in Figure~\ref{fig:visleaf}.

\begin{figure}[h]     
\begin{center}     
\begin{tikzpicture}     
	{          
	\draw[thick] (0,0)--(1,0);    
	\draw[fill] (0,0) circle (0.08);     
	\draw[fill] (1,0) circle (0.08);
	
	\draw[dashed] (2,0.8) arc (90:270:1.5cm and 0.8cm);
	
	\node at (0.4,0.8) {$T$};
	\node at (1.7,0.4) {$T_w$};
    \node at (0,0.3) {$v$};
	\node at (1,0.3) {$w$};
	\node at (1.5,0) {$\ldots$};    
	     } 	
\end{tikzpicture}\\ 
\end{center} 
\caption{Case 1: $v$ is a leaf of $T$.}
\label{fig:visleaf} 
\end{figure}
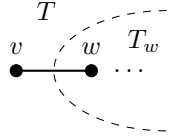 

Theorem \ref{thm:remove1/2node} tells us that $D_{T}\left(v,w\right)\geq D_{T}\left(w\right).$
Furthermore, by Theorem \ref{thm:DSvsDS-w}, we know that $D_{T}\left(v,w\right)\geq D_{T}\left(v\right)$
if and only if
\[
1-i\left(w;v\right)\leq D_{T}\left(v,w\right)=D_{T_{w}}\left(w\right),
\]
where the last equality follows from $\mu_{T}\left(v,w\right)=\mu_{T_{w}}\left(w\right)+1.$
That is to say, if $T_{w}$ is a tree of type H rooted at $w,$ then
$D_{T}\left(v,w\right)$ is no less than $D_{T}\left(v\right)$ and
$D_{T}\left(w\right).$ If $T_{w}$ is of type L, then $D_{T}\left(v,w\right)<D_{T}\left(v\right)$
and $D_{T}\left(v\right)$ is the largest of the three.

Now assume that both $v,w$ are not leaves, as in Figure~\ref{fig:visnotleaf}.
Let $T_{v}$ (resp. $T_{w}$) be the component of $v$ (resp. $w$)
in $T-vw,$ and $n_{v}=\left|T_{v}\right|$ (resp. $n_{w}=\left|T_{w}\right|$).
Then $D_{T_{v}}\left(v\right)$ and $D_{T_{w}}\left(w\right)$ are
both well defined. Denote the number of subtrees of $T_{v}$ rooted
at $v$ and the total cardinalities of such subtrees by $N_{v}$ and
$R_{v}.$ Denote the corresponding quantities of $w$ by $N_{w}$
and $R_{w}.$ 

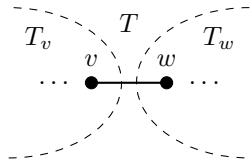
\begin{figure}[h]     
\begin{center}     
\begin{tikzpicture}     
	{          
	\draw[thick] (0,0)--(1,0);    
	\draw[fill] (0,0) circle (0.08);     
	\draw[fill] (1,0) circle (0.08);
	
	\node at (-0.5,0) {$\ldots$};
	
	\draw[dashed] (2.1,1) arc (90:270:1.5cm and 1cm);
	\draw[dashed] (-1.1,-1) arc (-90:90:1.5cm and 1cm);
	
	\node at (0.5,0.8) {$T$};
	\node at (1.7,0.6) {$T_w$};
	\node at (-0.7,0.6) {$T_v$};
    \node at (0,0.3) {$v$};
	\node at (1,0.3) {$w$};
	\node at (1.5,0) {$\ldots$};    
	     } 	
\end{tikzpicture}\\ 
\end{center} 
\caption{Case 2: neither $v$ nor $w$ is a leaf of $T$.}
\label{fig:visnotleaf} 
\end{figure} 

We start with an application of Theorem \ref{thm:DSvsDS-w},
\begin{align*}
\mu_{T}\left(v\right) & =\mu_{T}\left(v,w\right)-i\left(w;v\right)\\
\mu_{T}\left(v\right) & =\mu_{T_{v}}\left(v\right)+\mu_{T_{w}}\left(w\right)-i\left(w;v\right)\\
\mu_{T}\left(v\right)-1 & =\left(\mu_{T_{v}}\left(v\right)-1\right)+\left(\mu_{T_{w}}\left(w\right)-1\right)+\left(1-i\left(w;v\right)\right).
\end{align*}
Dividing both sides by $n-1,$ we have
\begin{align}
D_{T}\left(v\right) & =\frac{n_{v}-1}{n-1}D_{T_{v}}\left(v\right)+\frac{n_{w}-1}{n-1}D_{T_{w}}\left(w\right)+\frac{1}{n-1}\left(1-i\left(w;v\right)\right).\label{eq:Dv}
\end{align}
An analogous calculation provides
\begin{equation}
D_{T}\left(w\right)=\frac{n_{v}-1}{n-1}D_{T_{v}}\left(v\right)+\frac{n_{w}-1}{n-1}D_{T_{w}}\left(w\right)+\frac{1}{n-1}\left(1-i\left(v;w\right)\right).\label{eq:Dw}
\end{equation}
Before proceeding, we make some observations.
\begin{itemize}
\item Since
\[
\frac{n_{v}-1}{n-1}+\frac{n_{w}-1}{n-1}+\frac{1}{n-1}=1,
\]
and each term is in $\left[0,1\right],$ $D_{T}\left(v\right)$ is
a convex combination of $D_{T_{v}}\left(v\right),D_{T_{w}}\left(w\right)$
and $1-i\left(w;v\right),$ and the same holds for $D_{T}\left(w\right).$
\item $D_{T}\left(v,w\right)$ lies between $D_{T_{v}}\left(v\right)$ and
$D_{T_{w}}\left(w\right)$ as
\begin{align*}
D_{T}\left(v,w\right) & =\frac{\mu_{T}\left(v,w\right)-2}{n-2}\\
 & =\frac{\mu_{T_{v}}\left(v\right)-1+\mu_{T_{w}}\left(w\right)-1}{n_{v}-1+n_{w}-1},
\end{align*}
and both $\frac{\mu_{T_{v}}\left(v\right)-1}{n_{v}-1}=D_{T_{v}}\left(v\right)$
and $\frac{\mu_{T_{w}}\left(w\right)-1}{n_{w}-1}=D_{T_{w}}\left(w\right)$
are positive quantities $<1.$
\item The relation between $D_{T_{v}}\left(v\right)$ and $1-i\left(v;w\right)$
is determined $\mathit{only}$ by the structure of $T_{v},$ in other
words, the type of rooted tree $T_{v}.$ 
\item Theorem \ref{thm:DSvsDS-w} says that $D_{T}\left(v\right)\geq1-i\left(w;v\right)\Leftrightarrow D_{T}\left(v,w\right)\geq D_{T}\left(v\right)$
and $D_{T}\left(w\right)\geq1-i\left(v;w\right)\Leftrightarrow D_{T}\left(v,w\right)\geq D_{T}\left(w\right).$
Hence the relation of $D\left(v,w\right),$ $D\left(v\right)$ and
$D\left(w\right)$ is encoded in the relation of $D_{T_{v}}\left(v\right),D_{T_{w}}\left(w\right)$
and the types of $T_{v}$ and $T_{w}.$
\end{itemize}
We get back to the case of two-vertex subtrees. Assume $D_{T_{v}}\left(v\right)\geq D_{T_{w}}\left(w\right).$
Then we have the following results.

\begin{table}[H]
\captionsetup{justification=centering}

\caption{Relation of $D_{T}\left(v,w\right),D_{T}\left(v\right)$ and $D_{T}\left(w\right)$
under the assumption $D_{T_{v}}\left(v\right)\protect\geq D_{T_{w}}\left(w\right)$}

\begin{tabular}{|c|c|c|}
\cline{2-3} \cline{3-3} 
\multicolumn{1}{c|}{} & $T_{w}:$ type H & $T_{w}:$ type L\tabularnewline
\hline 
$T_{v}:$ type H & $D_{T}\left(v,w\right)\geq D_{T}\left(v\right),D_{T}\left(w\right)$ & no conclusion\tabularnewline
\hline 
$T_{v}:$ type L & $D_{T}\left(v\right)\geq D_{T}\left(v,w\right)\geq D_{T}\left(w\right)$ & $D_{T}\left(v\right)\geq D_{T}\left(v,w\right)$\tabularnewline
\hline 
\end{tabular}

\end{table}

\begin{thm}
Let $v,w$ be two adjacent vertices of tree $T.$ Let $T_{v},T_{w}$
be the subtrees of $T-vw$ rooted at $v$ and $w$ respectively. Then, 
\begin{enumerate}
\item if both $T_{v}$ and $T_{w}$ are of type H, then $D_{T}\left(v,w\right)\geq\max\left\{ D_{T}\left(v\right),D_{T}\left(w\right)\right\} $
\item if $T_{v}$ is of type L and $D_{T_{v}}\left(v\right)\geq D_{T_{w}}\left(w\right),$
then $D_{T}\left(v\right)>D_{T}\left(v,w\right).$
\end{enumerate}
\end{thm}

\begin{proof}
We shall derive the second case and the first is obtained similarly.
By assumption, we have
\[
D_{T_{v}}\left(v\right)<1-i\left(v;w\right).
\]
By equation (\ref{eq:Dw}), it implies
\[
D_{T}\left(w\right)<1-i\left(v;w\right).
\]
The last inequality is strict as singleton rooted trees are of type
H, hence we know $n_{v}>1$ which implies that the first coefficient
$\frac{n_{v}-1}{n-1}$ is positive. Then, it is further equivalent
to
\[
D_{T}\left(v,w\right)<D_{T}\left(v\right)
\]
by Theorem \ref{thm:DSvsDS-w}.
\end{proof}
From the last theorem and the results on types of rooted trees, one
readily obtains the following more intuitive description.
\begin{cor}
Let notations be as above. Then, 
\begin{enumerate}
\item $\deg_{T}\left(v\right)=\deg_{T}\left(w\right)=2$ is a necessary
condition for $D_{T}\left(v,w\right)\geq D_{T}\left(v\right)$ and
$D_{T}\left(v,w\right)\geq D_{T}\left(w\right);$
\item if $\deg_{T}\left(v\right)>2$ and $\deg_{T}\left(w\right)>2,$ then
$D_{T}\left(v,w\right)<\max\left\{ D_{T}\left(v\right),D_{T}\left(w\right)\right\} .$
\end{enumerate}
\end{cor}

\section{Acknowledgement}

I thank Professor S. Wagner for all the discussions and advice.

\bibliographystyle{plain}
\bibliography{Localmeananddensity}

\begin{thebibliography}{1}

\bibitem{MR4282633}
Stijn Cambie, Stephan Wagner, and Hua Wang.
\newblock On the maximum mean subtree order of trees.
\newblock {\em European J. Combin.}, 97:Paper No. 103388, 19, 2021.

\bibitem{MR3213626}
John Haslegrave.
\newblock Extremal results on average subtree density of series-reduced trees.
\newblock {\em J. Combin. Theory Ser. B}, 107:26--41, 2014.

\bibitem{MR735190}
Robert~E. Jamison.
\newblock On the average number of nodes in a subtree of a tree.
\newblock {\em J. Combin. Theory Ser. B}, 35(3):207--223, 1983.

\bibitem{MR762896}
Robert~E. Jamison.
\newblock Monotonicity of the mean order of subtrees.
\newblock {\em J. Combin. Theory Ser. B}, 37(1):70--78, 1984.

\bibitem{MR4563206}
Zuwen Luo, Kexiang Xu, Stephan Wagner, and Hua Wang.
\newblock On the mean subtree order of trees under edge contraction.
\newblock {\em J. Graph Theory}, 102(3):535--551, 2023.

\bibitem{MR3982896}
Lucas Mol and Ortrud~R. Oellermann.
\newblock Maximizing the mean subtree order.
\newblock {\em J. Graph Theory}, 91(4):326--352, 2019.

\bibitem{MR2595700}
Andrew Vince and Hua Wang.
\newblock The average order of a subtree of a tree.
\newblock {\em J. Combin. Theory Ser. B}, 100(2):161--170, 2010.

\bibitem{MR3433637}
Stephan Wagner and Hua Wang.
\newblock On the local and global means of subtree orders.
\newblock {\em J. Graph Theory}, 81(2):154--166, 2016.

\end{thebibliography}

\end{document}